\documentclass[a4paper,11pt]{article}
\usepackage[
    a4paper,
    top=1.05in,
    bottom=1.05in,
    left=1.00in,
    right=1.00in
]{geometry}
\usepackage[T1]{fontenc}
\usepackage{amsmath,amsthm,mathtools}
\usepackage{amssymb}
\usepackage{newtxtext,newtxmath}
\usepackage{microtype}
\usepackage{iftex}
\ifXeTeX
  \microtypesetup{protrusion=true,expansion=false}
\else
  \microtypesetup{protrusion=true,expansion=true}
\fi
\usepackage[new]{old-arrows}
\usepackage{mathrsfs}
\allowdisplaybreaks[4]
\usepackage{enumitem}
\usepackage{needspace}
\setlist{itemsep=0.25em, topsep=0.45em, parsep=0pt}

\usepackage{titlesec}
\titleformat{\section}{\normalfont\Large\bfseries}{\thesection}{0.75em}{}
\titleformat{\subsection}{\normalfont\large\bfseries}{\thesubsection}{0.75em}{}
\titlespacing*{\section}{0pt}{2.6ex plus 0.8ex minus 0.2ex}{1.2ex plus 0.2ex}
\titlespacing*{\subsection}{0pt}{2.0ex plus 0.6ex minus 0.2ex}{0.9ex plus 0.2ex}

\usepackage{fancyhdr}
\fancypagestyle{plain}{%
  \fancyhf{}%
  \fancyfoot[C]{\thepage}%
}

\usepackage[dvipsnames]{xcolor}
\definecolor{linkblue}{RGB}{20,67,128}
\definecolor{citegreen}{RGB}{25,105,80}
\definecolor{urlpurple}{RGB}{100,55,125}
\usepackage[colorlinks=true,
            linkcolor=linkblue,
            citecolor=citegreen,
            urlcolor=urlpurple,
            plainpages=false]{hyperref}
\usepackage{bookmark}

\numberwithin{equation}{section}
\theoremstyle{plain}
\newtheorem{theorem}{Theorem}[section]
\newtheorem{proposition}[theorem]{Proposition}
\newtheorem{lemma}[theorem]{Lemma}
\newtheorem{corollary}[theorem]{Corollary}
\theoremstyle{remark}
\newtheorem{remark}[theorem]{Remark}
\theoremstyle{plain}
\renewcommand{\d}{\mathrm{d}}
\newcommand{\dc}{\mathrm{d}^{c}}
\newcommand{\ddc}{\mathrm{d}\mathrm{d}^{c}}
\newcommand{\heat}{\square}
\newcommand{\muent}{\boldsymbol{\mu}}
\newcommand{\scal}{\mathrm{R}}
\newcommand{\RR}{\mathcal{R}}
\newcommand{\NN}{\mathcal{N}}
\newcommand{\WW}{\mathcal{W}}
\newcommand{\cG}{\mathcal{G}}
\DeclareMathOperator{\Ric}{Ric}
\DeclareMathOperator{\Id}{Id}
\DeclareMathOperator{\Vol}{Vol}

\hypersetup{pdftitle={Scalar curvature bounds and tangent cones for Kahler--Ricci shrinkers},
  pdfauthor={Yu Li and Junsheng Zhang}}
\title{Scalar curvature bounds and tangent cones\\for K\"ahler--Ricci shrinkers}
\author{Yu Li \quad and \quad Junsheng Zhang}
\date{}

\begin{document}
\maketitle
\begin{abstract}
We prove a uniform scalar curvature bound, depending only on the 
dimension, for K\"ahler--Ricci shrinkers.
We also prove that every noncompact K\"ahler--Ricci shrinker has a unique
tangent cone at infinity, naturally homeomorphic to the base of its
polarized Fano fibration.
For shrinkers with maximal volume growth, this cone is
isometric to the metric completion of the regular part equipped with the intrinsic distance induced by
the limiting K\"ahler metric.
\end{abstract}

\begingroup
\small
\tableofcontents
\endgroup
\clearpage

\section{Introduction}\label{sec:introduction}
\addtocontents{toc}{\protect\setcounter{tocdepth}{1}}

A K\"ahler--Ricci shrinker $(X^n,J,g,\omega,f)$ is a complete shrinking
gradient K\"ahler--Ricci soliton. Throughout this paper, we use the complex
normalization
\begin{equation}\label{eq:shrinker}
\begin{gathered}
    \Ric(\omega)+\sqrt{-1}\partial\bar\partial f=\omega,
    \qquad \Delta_\omega f+f-|\nabla^{1,0}f|_\omega^2=0,\\
    \rho^2:=f+n=\scal_\omega+|\nabla^{1,0}f|_\omega^2,
\end{gathered}
\end{equation}
where $n$ is the complex dimension.
The scalar curvature is nonnegative by \cite[Corollary~2.5]{Chen2009Strong},
and the soliton potential $f$ has quadratic growth by
\cite[Theorem~1.1]{CaoZhou} and \cite[Lemma~2.1]{HaslhoferMuller}.

\subsection{Scalar curvature bounds}
The soliton identities give the general bound
$0\leq\scal_\omega\leq\rho^2$, so the scalar curvature grows at most
quadratically in the distance from a fixed point. Improving this to a
uniform bound is important for understanding the geometry and
classification of shrinkers. In complex dimension two, Li--Wang
\cite[Theorem~1.2]{LW26} proved that every K\"ahler--Ricci shrinker has
bounded scalar curvature. Together with the four-dimensional curvature
estimate of Munteanu--Wang \cite[Theorem~1.1]{MunteanuWang2015}, this
implies bounded sectional curvature. Combined with the construction and
classification results of Conlon--Deruelle--Sun \cite{ConlonDeruelleSun},
Cifarelli--Conlon--Deruelle \cite{CifarelliConlonDeruelle}, and
Bamler--Cifarelli--Conlon--Deruelle \cite{BCCD}, it completes the
classification of noncompact K\"ahler--Ricci shrinker surfaces.

For compact K\"ahler--Ricci shrinkers in arbitrary dimension,
Guo--Phong--Song--Sturm \cite{GPSS} proved upper bounds for the
scalar curvature and diameter depending only on the dimension.
For noncompact shrinkers, Esparza and the second-named author
\cite[Theorem~3.1]{EsparzaZhang} obtained a uniform subquadratic bound
for the scalar curvature. Our first main result is a
uniform bound on the scalar curvature depending only on the dimension.

\begin{theorem}\label{thm:main}
There is a constant $C_n<\infty$ such that every K\"ahler--Ricci shrinker of complex dimension $n$
satisfies
\begin{equation}\label{eq:main-bound}
    0\le\scal_\omega\le C_n.
\end{equation}
\end{theorem}

A further motivation comes from the recent work of Jian--Song
\cite{JS26}, who proved a Type~I scalar curvature bound for every
finite-time singularity of the K\"ahler--Ricci flow on a compact
K\"ahler manifold. This extends Perelman's scalar curvature estimate
for the K\"ahler--Ricci flow starting from the first Chern class of a Fano manifold;
see \v{S}e\v{s}um--Tian \cite{SesumTian}. Our earlier work
\cite{LZ26} developed an analogy between the geometry of compact Ricci
flows near singular times and the geometry of noncompact shrinkers
at infinity.
For the canonical ancient flow $\omega_t=(-t)\psi_t^*\omega$ induced by a
K\"ahler--Ricci shrinker, Theorem~\ref{thm:main} gives the Type~I
scalar curvature bound
\[
    0\leq\scal_{\omega_t}=(-t)^{-1}\psi_t^*\scal_\omega
       \leq\frac{C_n}{|t|},\qquad t<0.
\]

The proof adapts the one-form method and iteration of the energy estimate
over scales in
\cite[Sections~3, 6, and~7]{JS26} to the canonical ancient flow.
The shrinker provides a natural global potential $F$, whose exact
evolution and Hessian identities are recorded in
Lemma~\ref{lem:canonical}. We introduce the energy $\mathcal D$ to measure
how far $\dc F$ is from the weighted $L^2$ closure of gradients;
see \eqref{eq:D-definition}.
Since $\nabla^2F$ is $J$-invariant, $\ddc F$ determines its full
Hessian. The soliton identities then relate this energy to the
scalar curvature; see Proposition~\ref{prop:scalar-reduction}.

The noncompact setting requires heat kernel estimates and compactness
without an a priori curvature bound. We use Li--Wang's heat kernel,
compactness, and regularity results for induced shrinker flows
\cite{liwang2020heat1,liwang2024heat}, which extend Bamler's
metric-flow theory \cite{bamler2023compactness,bamler3}.
The scalar curvature bound then allows us to apply the spacetime
convergence and regularity theory of Fang--Li \cite{FangLiStructure},
in the complete-shrinker setting of \cite[Section~5]{LZ26}.
At scales with small pointed Nash entropy drop, compactness and a
weighted Bochner vanishing argument give a stronger contraction
estimate for $\mathcal D$ between scales; see
Proposition~\ref{prop:improved-contraction}.
We iterate over larger backward scales. The ancient flow and the
quadratic energy bound \eqref{eq:coarse-growth} allow the terminal
term to tend to zero, yielding a scalar curvature estimate depending
only on the dimension and a lower bound for Perelman's entropy.

With the normalization \eqref{eq:shrinker}, Perelman's entropy
\cite[Section~3.1]{Perelman2002} is given by
\cite{CarrilloNi2009,liwang2020heat1}
\begin{equation}\label{eq:entropy}
    \muent(g)=\mu(g,\tfrac12)
      =\log\left((2\pi)^{-n}\int_X e^{-(f+n)}\frac{\omega^n}{n!}\right).
\end{equation}
The second-named author \cite[Theorem~0.1]{ZhangUniform} proved that every
K\"ahler--Ricci shrinker satisfies
\begin{equation}\label{eq:automatic-entropy}
    \muent(g)\geq-A_n,
\end{equation}
where $A_n$ depends only on $n$. This controls the number of scales
at which the improved contraction may fail and makes the scalar
curvature bound depend only on the dimension.

\subsection{Tangent cones at infinity}

One of the main applications of Theorem~\ref{thm:main} is to the
geometry at infinity of K\"ahler--Ricci shrinkers, continuing the
program in \cite{LZ26}. For a shrinker, self-similarity relates the tangent space at infinity
to the singular-time limit of its canonical flow.
In \cite[Theorem~1.6]{LZ26}, we proved existence and uniqueness of
this tangent space under two assumptions: bounded scalar curvature
and local compactness of the canonical slice $(Z_0,d_0^Z)$.
Theorem~\ref{thm:main} makes the first assumption automatic in the
K\"ahler setting. We establish the second in
Section~\ref{sec:time-zero-fibration} and identify the limit with the
base of the polarized Fano fibration.

Let $\pi:X\to Y$ be the polarized Fano fibration associated with a
noncompact K\"ahler--Ricci shrinker
\cite[Theorem~3.1]{SunZhang}, and let $0\in Y$ be its vertex.
Our second main result also establishes the metric cone structure
of the limit, proving part of the asymptotic conicity conjecture
\cite[Conjecture~6.2]{SunZhang}.

\begin{theorem}\label{thm:canonical-cone}
Let $(X^n,\omega,f)$ be a noncompact K\"ahler--Ricci shrinker.
There is a proper geodesic metric $d_Y$ on $Y$, inducing its
analytic topology, such that for every $p\in X$,
\begin{equation}\label{eq:intro-cone-limit}
    (X,r^{-2}g,p)\longrightarrow(Y,d_Y,0)\cong C(\Sigma)
    \quad\text{as }r\to\infty
\end{equation}
in the pointed Gromov--Hausdorff sense, where $(\Sigma,d_\Sigma)$ is a
compact geodesic space. In particular, the tangent cone at infinity is unique.
\end{theorem}

In
Section~\ref{sec:time-zero-fibration}, we first derive polynomial
control of the evolving metrics from \cite[Lemma~5.4]{ZhangUniform}
and \cite[Proposition~4.4]{LZ26}. The potential identities then give
uniform local H\"older bounds for K\"ahler potentials. The distance
estimates of Guedj--Guenancia--Zeriahi
\cite[Proposition~1.4]{GGZ}, following the method of
Yang Li \cite[Theorem~4.1]{YangLi}, then give local H\"older bounds
for distances. The worldline map $\Phi:X\to Z_0$ is therefore continuous
for $d_0^Z$, and the proper limiting potential makes the slice
proper. This proves the required local compactness.
Next, the Dirichlet energy of distance functions along rational
curves in each Fano fiber tends to zero. Rational chain
connectedness of the fibers, due to Hacon--McKernan
\cite[Corollary~1.4]{HaconMcKernan}, shows that $\Phi$ is constant on every
fiber, and the comparison map of \cite{LZ26} identifies $Z_0$ with
$Y$. Under this identification, $\Phi$ realizes the polarized Fano
fibration, proving \cite[Conjecture~6.7]{LZ26}. We also obtain a
dimensional bound for the ambient diameters of all fibers; see
Corollary~\ref{cor:uniform-fiber-diameter}.

The remaining step is to determine the metric cone structure.
The results of \cite[Section~5]{LZ26}, based on the spacetime
structure theory of \cite{FangLiStructure}, already give a
topological cone with a natural family of homotheties.
For manifolds with nonnegative Ricci curvature and Euclidean volume
growth, Cheeger--Colding \cite{CheegerColding} proved that tangent
spaces at infinity are metric cones. Here the Ricci curvature need
not be nonnegative, even on the tangent cone at infinity
\cite{DancerWang}.
The volume growth need not be Euclidean either.
We instead use the weighted Ricci identity
\[
    \Ric(g_{-\tau})+\nabla^2\bigl(F(\cdot,-\tau)/\tau\bigr)
       =\tau^{-1}g_{-\tau}.
\]
The entropy gradient-flow theory of Ambrosio--Gigli--Savar\'e
\cite{AGS} then yields an evolution variational inequality for the
scaled relative entropy. As $\tau\to0$, the weighted heat flow from $x$
concentrates at $D_{e^{-s}}x$; see Lemma~\ref{lem:tangent-diffusion}.
Passing to the limit gives the radial evolution inequality in
Proposition~\ref{prop:tangent-radial-evi}. Interchanging the two points
and using homothety gives the reverse inequality. The resulting
equality integrates to the cone distance formula.
This argument also applies to real Ricci shrinkers with
bounded scalar curvature and locally compact time-zero slice;
see Theorem~\ref{thm:real-shrinker-cone}.

\subsection{Metric completion of the regular part}\label{subsec:intro-metric-completion}
Theorem~\ref{thm:canonical-cone} identifies the tangent cone at infinity
as a metric space. We now compare this metric with the smooth limit
of the canonical flow. Suppose that $(X^n,\omega,f)$ has maximal volume
growth. Then its polarized Fano fibration $\pi:X\to Y$ is birational
by \cite[Proposition~7.1]{LZ26}. Let $Y^\circ\subset Y$ be the largest
open set over which $\pi$ is an isomorphism, and put
$X^\circ=\pi^{-1}(Y^\circ)$. The canonical flow converges locally
smoothly on $X^\circ$ and, through $\pi$, defines a K\"ahler metric
$g_0$ on $Y^\circ$; see \cite[Corollary~4.5 and Theorem~6.2]{LZ26}.

The ambient distance $d_Y$ and the intrinsic Riemannian distance
$d_{g_0}$ on $Y^\circ$ satisfy $d_Y\leq d_{g_0}$ by smooth
convergence. 
Our third main result establishes the reverse equality and determines the
whole cone by metric completion. Together with Theorem~\ref{thm:canonical-cone}, this proves
\cite[Conjecture~6.8]{LZ26}.

\begin{theorem}[Metric completion]\label{thm:completion}
Let $(X^n,\omega,f)$ be a K\"ahler--Ricci shrinker with maximal volume
growth, and let $d_Y$ be the metric of
Theorem~\ref{thm:canonical-cone}. The inclusion
$Y^\circ\hookrightarrow Y$ extends to an isometry
\[
    \overline{(Y^\circ,d_{g_0})}\longrightarrow(Y,d_Y),
\]
where $d_{g_0}$ is the intrinsic Riemannian distance on $Y^\circ$.
\end{theorem}

The key is to approximate a $d_Y$-geodesic segment
joining regular points by curves in $X^\circ\simeq Y^\circ$, with
arbitrarily small length error. To construct these curves, we blow up
at almost every interior point of the geodesic  and use the tangent-flow framework in 
\cite{FangLiStructure}, applied to the canonical
shrinker flow as in \cite[Section~5]{LZ26}. A similar idea of using
tangent geometry to construct curves in the regular part appears in
Chen--Wang \cite[Theorem~3.18 and Proposition~3.20]{CWIIb}.

We first show that the rescaled geodesic
converges to a geodesic line in the time-zero slice of a tangent flow.
Then in Lemma~\ref{curve:orbit}, we prove that this line induces a flat complex splitting
of the regular backward tangent flow. We transfer holomorphic strips in this factor to the approximating
flows. The scalar curvature bound in
Theorem~\ref{thm:main} gives the uniform estimate
$|F(\cdot,-\tau)-F_0|\leq C_n\tau$ for the canonical potential.
Through the soliton equation, this controls the area of the
strip. Averaging over their horizontal segments then produces the desired
short curves in $Y^\circ$.

\subsection*{Organization}
Section~\ref{sec:heat} collects the heat kernel estimates, introduces the
weighted one-form energy, and reduces the scalar curvature bound to an
estimate for this energy. Section~\ref{sec:scalar-bound} proves
Theorem~\ref{thm:main}: Subsection~\ref{subsec:almost-selfsimilar}
regularizes heat-kernel potentials near almost self-similar points,
Subsection~\ref{subsec:vanishing} proves the vanishing lemma for one-forms,
and Subsection~\ref{subsec:energy-iteration} improves the energy estimate
and completes the proof by iteration.
Section~\ref{sec:time-zero-fibration} identifies the time-zero slice
with the base of the Fano fibration. Section~\ref{sec:metric-cone} proves its metric
cone structure Theorem~\ref{thm:canonical-cone}.
Section~\ref{sec:metric-completion} proves the metric-completion
Theorem~\ref{thm:completion}. In Appendix~\ref{app:finite-time}, we prove the metric completion result
for finite-time singularities of compact volume-noncollapsed
K\"ahler--Ricci flows.

\subsection*{Notations and conventions}
We denote the underlying manifold of a K\"ahler--Ricci shrinker by $X$,
its complex dimension by $n$, its complex structure by $J$,
and the Riemannian metric associated with $\omega$ by $g$; thus
$\d V_g=\omega^n/n!$. We use
$\Delta=\Delta_\omega=\tfrac12\Delta_g$ and
$\scal=\scal_\omega=\tfrac12\scal_g$ unless a metric is specified.
In tensor equations, $\Ric$ denotes the real Ricci tensor $\Ric_g$,
whereas $\Ric(\omega)$ denotes the Ricci form.
Gradients, Hessians, and norms of real tensors are taken with respect to $g$.

We use $\dc u=-\d u\circ J$ and the nonnegative Hodge Laplacian
$\Delta_\d=\d\delta+\delta\d$, where $\delta$ is the adjoint of $\d$
with respect to $\d V_g$. For a weight $b$,
$\delta_b=\delta+\iota_{\nabla b}$ is the adjoint of $\d$ in
$L^2(e^{-b}\d V_g)$.

Our flow convention is $\partial_t\omega_t=-\Ric(\omega_t)$,
or equivalently $\partial_tg=-\Ric_g$, and
$\heat=\partial_t-\Delta$. Following
\cite[Definition~2.2]{liwang2024heat}, a flow induced by a
shrinker is a time translation and parabolic rescaling
of its canonical Ricci flow.
The results of \cite{liwang2020heat1,liwang2024heat} apply to $2g(t)$,
which satisfies $\partial_t(2g)=-2\Ric_{2g}$; the entropy and conjugate heat kernel measures are unchanged by this change of convention.

For a base point $x_0^*=(x_0,t_0)$, we write
$\tau=t_0-s$ for backward time and $\nu_s$ for the conjugate heat
kernel measure based at $x_0^*$.
Unless otherwise indicated,
spatial norms $\|\cdot\|_p$ are taken in $L^p(\nu_s)$ with respect to the
metric at the time under consideration. We write $d_{W_1}^{g(t)}$ for
the Wasserstein $1$-distance with ground distance induced by $g(t)$;
on a metric slice $Z_t$ with canonical distance $d_t^Z$, we write $d_{W_1}^{Z_t}$.

\subsection*{Acknowledgements}
Yu Li is supported by the National Key R\&D Program of China
(Grant No.~2025YFA1018200), NSFC-12522105, and research funds from the
University of Science and Technology of China and the Chinese Academy of
Sciences.

\subsection*{AI declaration}
ChatGPT (GPT-6 Astra) was used to assist with drafting, mathematical
exposition, revision, and reference checking. The authors take full
responsibility for the mathematical content and correctness.

\addtocontents{toc}{\protect\setcounter{tocdepth}{\number\value{tocdepth}}}
\section{Heat kernel and weighted one-form estimates}\label{sec:heat}

\subsection{Conjugate heat kernel measures and the soliton potential}

Fix a flow $g(t)$, $t<T$, obtained from the canonical flow of a K\"ahler--Ricci shrinker
by time translation and parabolic rescaling, and a spacetime point
$x_0^*=(x_0,t_0)$ with $t_0<T$. For $s<t_0$, write
\begin{equation}\label{eq:heat-measure}
    \d\nu_s=K(x_0,t_0;\cdot,s)\,\d V_{g(s)}
      =(2\pi\tau)^{-n}e^{-b(s)}\,\d V_{g(s)},
      \qquad \tau=t_0-s.
\end{equation}
Here $\nu_s=\nu_{x_0^*;s}$ is the conjugate heat kernel measure and
$b=b_{x_0^*}$ is its potential at $x_0^*$, in the terminology of
\cite[Definition~2.6]{FangLiStructure}. We suppress the basepoint when it
is fixed. Each $\nu_s$ is a probability measure. For $a<a'$, the operator $P_{a,a'}$
evolves a function from time $a$ to time $a'$ by the scalar heat equation:
\begin{equation}\label{eq:propagator-definition}
    (P_{a,a'}u_a)(z)=\int_X K(z,a';y,a)u_a(y)\,\d V_{g(a)}(y).
\end{equation}
Thus $u(\cdot,s)=P_{a,s}u_a$ solves
$\partial_su=\Delta_{\omega_s}u$ with initial value $u(\cdot,a)=u_a$.
We first define it for compactly supported smooth $u_a$ and then extend it
to the weighted spaces below. If $a<a'<t_0$, the semigroup property gives
\begin{equation}\label{eq:duality}
    \int_X P_{a,a'}u\,\d\nu_{a'}=\int_X u\,\d\nu_a,
    \qquad a<a'<t_0.
\end{equation}
For a smooth function $u$ satisfying the integrability conditions of
Lemma~\ref{lem:cutoff-duality}, the corresponding differentiated identity holds
for almost every $s$:
\begin{equation}\label{eq:heat-duality}
    \frac{\d}{\d s}\int_X u\,\d\nu_s
       =\int_X\heat u\,\d\nu_s.
\end{equation}
To construct cutoffs, fix a compact time interval $I\Subset(-\infty,T)$
and set
\[
    \tilde F(x,s)=\lambda F\bigl(x,(s-T)/\lambda\bigr),\qquad\lambda>0,
\]
where $\lambda$ is the parabolic scaling factor of the induced flow
and $F$ is the canonical potential in \eqref{eq:canonical-potential}.
In our convention, the
potential identities in \cite[p.~1644, (2.5)--(2.9)]{liwang2024heat}
take the form
\[
    (\partial_s-\Delta_{\omega_s})\tilde F=-n,
    \qquad
    \tfrac12|\d\tilde F|_{g(s)}^2
       =\tilde F-(T-s)^2\scal_{\omega_s}.
\]
Choose a smooth function
$0\le\eta\le1$, equal to one on $[0,1]$ and zero on $[2,\infty)$, and
set
\begin{equation}\label{eq:chiL}
    \chi_L=\eta(\tilde F/L),\qquad L\ge1.
\end{equation}
These are compactly supported in space, exhaust the manifold uniformly
on the fixed time interval, and satisfy
\begin{equation}\label{eq:cutoff-bounds}
    |\d\chi_L|\le C L^{-1/2},\qquad
    |\partial_s\chi_L|+|\Delta\chi_L|\le C_I.
\end{equation}
The derivatives are supported where $L\le\tilde F\le2L$.
These estimates follow from the canonical-potential identities and
$0\le\scal\le\tilde F/(T-s)^2$; they are also the cutoffs of
\cite[(2.11)--(2.15)]{liwang2024heat}. The constants may depend on the fixed flow and time interval.

\begin{lemma}\label{lem:cutoff-duality}
Let $\d\nu_s=K_s\d V_{g(s)}$ be a conjugate heat kernel measure and
$K_s=(2\pi(t_0-s))^{-n}e^{-b(s)}$. Fix
$[a,a']\Subset(-\infty,t_0)$ and a smooth scalar function $B$ such that
$B(a)\in L^1(\nu_a)$ and $B,\,B|\d b|\in L^1(\d\nu_s\,\d s)$
on $X\times[a,a']$. If $\heat B=G$ with
$G\in L^1(\d\nu_s\,\d s)$, then $B(a')\in L^1(\nu_{a'})$ and
\begin{equation}\label{eq:integrated-heat-duality}
    \int_X B(a')\,\d\nu_{a'}-\int_X B(a)\,\d\nu_a
       =\int_a^{a'}\!\int_X G\,\d\nu_s\,\d s.
\end{equation}
We also have
\begin{equation}\label{eq:L1-Kato}
    \|B(a')\|_{L^1(\nu_{a'})}\le\|B(a)\|_{L^1(\nu_a)}+
                       \int_a^{a'}\!\int_X|G|\,\d\nu_s\,\d s.
\end{equation}
If instead $\heat B\le-Q$ with $B,Q\ge0$, then
\[
    \int_X B(a')\,\d\nu_{a'}
       +\int_a^{a'}\!\int_X Q\,\d\nu_s\,\d s
       \le\int_X B(a)\,\d\nu_a.
\]
\end{lemma}

\begin{proof}
This is the cutoff argument in \cite[Proposition~3.3(ii) and the proof of
Lemma~A.1, (A.3)--(A.4)]{liwang2024heat}, with an inhomogeneous term.
Differentiate using $\partial_sK=-\Delta K+\scal K$ and
$\partial_s\d V=-\scal\d V$, and integrate the Laplacian twice on the
compact support of $\chi_L$. This gives
\[
\frac{\d}{\d s}\int_X\chi_LB\,\d\nu_s
       =\int_X\chi_L\heat B\,\d\nu_s
       +\int_XB\bigl(\partial_s\chi_L+\Delta\chi_L
          -\langle\d b(s),\d\chi_L\rangle\bigr)\,\d\nu_s.
\]
The cutoff bounds and the integrability
of $B$ and $B|\d b|$ make the error tend to zero after time integration.

Apply the same calculation to $(B^2+\delta^2)^{1/2}-\delta$, whose heat operator
is at most $|G|$. Taking the initial boundary value before exhaustion,
then letting $L\to\infty$ and $\delta\to0$, proves
\eqref{eq:L1-Kato} by Fatou's lemma. This also gives terminal
integrability; \eqref{eq:integrated-heat-duality} then follows by dominated convergence.
For $\heat B\le-Q$ with $B,Q\ge0$, apply Fatou's lemma directly to the
nonnegative terminal and spacetime terms.
\end{proof}

Define
\begin{equation}\label{eq:N-W}
\begin{aligned}
    \NN_{x_0^*}(\tau)&=\int_X b(t_0-\tau)\,\d\nu_{t_0-\tau}-n,\\
    \WW_{x_0^*}(\tau)&=\int_X
       \bigl[\tau(\scal+\tfrac12|\d b|^2)+b-2n\bigr]
       \,\d\nu_{t_0-\tau}.
\end{aligned}
\end{equation}
Using \eqref{eq:automatic-entropy}, the estimates of
\cite[Proposition~3.19 and Corollary~3.22]{liwang2024heat} give
\begin{equation}\label{eq:entropy-bounds}
    -A_n\le\muent(g)\le\NN_{x_0^*}(\tau)\le0,
    \qquad
    \WW_{x_0^*}(\tau)=\frac{\d}{\d\tau}\bigl(\tau\NN_{x_0^*}(\tau)\bigr).
\end{equation}
Both $\WW_{x_0^*}$ and $\NN_{x_0^*}$ are nonincreasing in $\tau$.
For these entropy identities and monotonicity on compact Ricci flows,
see also \cite[Proposition~5.2]{BamlerHeat}. The differential
Harnack inequality \cite[Theorem~21]{liwang2020heat1} is
\begin{equation}\label{eq:Perelman-w}
    w(b):=\tau(\scal+2\Delta b-\tfrac12|\d b|^2)+b-2n\le0,
    \qquad \int_X w(b)\,\d\nu_{t_0-\tau}=\WW_{x_0^*}(\tau).
\end{equation}
We will also use \cite[Corollary~3.20, (3.54)]{liwang2024heat}:
\begin{equation}\label{eq:scalar-average}
    \int_X\bigl(\scal+\tfrac12|\d b|^2\bigr)\,\d\nu_{t_0-\tau}
       \le\frac n\tau.
\end{equation}

\begin{lemma}\label{lem:weighted-heat}
Let $a<a'<t_0$. Scalar heat evolution extends to a contraction
$L^p(\nu_a)\to L^p(\nu_{a'})$ for $1\le p<\infty$. For weighted Sobolev
initial data, it satisfies
\begin{equation}\label{eq:gradient-contraction}
    \|\d P_{a,a'}u\|_{L^p(\nu_{a'})}
       \le\|\d u\|_{L^p(\nu_a)},\qquad 2\le p<\infty,
\end{equation}
and
\begin{equation}\label{eq:hessian-energy}
    \|\d P_{a,a'}u\|_2^2
       +\int_a^{a'}\!\int_X|\nabla^2P_{a,s}u|^2\,\d\nu_s\,\d s
       \le\|\d u\|_2^2.
\end{equation}
Moreover,
\begin{equation}\label{eq:Poincare}
    \int_X\left|u-\int_X u\,\d\nu_a\right|^2\d\nu_a
       \le(t_0-a)\int_X|\d u|^2\,\d\nu_a.
\end{equation}
If $z\ge0$ is a heat subsolution of polynomial spatial growth, then
\begin{equation}\label{eq:hypercontractivity}
    \|z(a')\|_{L^r(\nu_{a'})}\le\|z(a)\|_{L^p(\nu_a)}
    \quad\text{provided}\quad
    1<p\le r<\infty,\quad
    \frac{t_0-a}{t_0-a'}\ge\frac{r-1}{p-1}.
\end{equation}
\end{lemma}

\begin{proof}
First take $u\in C_c^\infty(X)$ and put $U(s)=P_{a,s}u$.
Since the heat kernel measures are probability measures, Jensen's
inequality and \eqref{eq:duality} give, for $1\le p<\infty$,
\[
    \int_X|P_{a,a'}u|^p\,\d\nu_{a'}
       \le\int_X P_{a,a'}(|u|^p)\,\d\nu_{a'}
       =\int_X|u|^p\,\d\nu_a.
\]
By density, $P_{a,a'}$ extends to a contraction from $L^p(\nu_a)$
to $L^p(\nu_{a'})$. In our normalization, the chain rule and
Bochner formula give
\[
    \heat|U|^p=-\tfrac{p(p-1)}2|U|^{p-2}|\d U|^2\le0,
    \qquad \heat|\d U|^p\le0\quad(p\ge2).
\]
At zeros, use $(U^2+\delta^2)^{p/2}$ and
$(|\d U|^2+\delta^2)^{p/2}$ and let $\delta\to0$.
The gradient bound of \cite[Lemma~3.2]{liwang2024heat} makes $|\d U|$
bounded. The maximum principle \cite[Theorem~6]{liwang2020heat1},
recalled in \cite[Theorem~2.5]{liwang2024heat}, therefore applies and gives
\[
    |\d P_{a,a'}u|^p\le P_{a,a'}(|\d u|^p).
\]
Integrating against $\nu_{a'}$ proves \eqref{eq:gradient-contraction}.
For \eqref{eq:hessian-energy}, retain the full identity
\[
    \heat|\d U|^2=-|\nabla^2U|^2.
\]
Integrate with the cutoffs of Lemma~\ref{lem:cutoff-duality} and use
Fatou's lemma for the nonnegative Hessian term. The cutoff errors vanish
because $|\d U|$ is bounded and $|\d b|$ is integrable on each fixed
spacetime interval. Lemma~\ref{lem:approximation} below extends both
estimates from compactly supported smooth data to weighted Sobolev data.

The Poincar\'e inequality \eqref{eq:Poincare} follows from the logarithmic
Sobolev inequality \cite[Theorem~13]{liwang2020heat1}; it is stated in
\cite[Theorem~3.17]{liwang2024heat} with $p=2$.
Its coefficient is $t_0-a$, since
$2|\d u|_{2g}^2=|\d u|_g^2$.
The hypercontractivity estimate \eqref{eq:hypercontractivity} is
\cite[Theorem~4.20]{liwang2024heat}; see also
\cite[Theorem~12.1]{BamlerHeat} for compact Ricci flows.
\end{proof}

For a fixed conjugate heat kernel measure $\nu_a$ and $1\le p<\infty$, the weighted
Sobolev norm is
\[
    \|u\|_{W^{1,p}(\nu_a)}
       =\|u\|_{L^p(\nu_a)}+\|\d u\|_{L^p(\nu_a)}.
\]
We first define the space by completion of $C_c^\infty(X)$. It agrees
with the space of functions with weak derivative and finite displayed
norm. The following lemma constructs heat solutions with initial data
in $W^{1,p}(\nu_a)$ by approximation.

\begin{lemma}\label{lem:approximation}
Let $a<a'<t_0$ and $2\le p<\infty$.
Every $u\in W^{1,p}(\nu_a)$ admits compactly supported smooth
approximations $u_k$ in that norm. Scalar heat evolution of $u_k$
converges in $L^p(\nu_s)$ and its gradients converge in $L^p(\nu_s)$,
uniformly for $s\in[a,a']$. For $p=2$, the Hessians also converge in
spacetime $L^2$ on $[a,a']$. The limit is the weighted heat evolution of
$u$, is smooth locally for $s>a$, and belongs to
$W^{1,p}(\nu_s)$ at each later time.
\end{lemma}

\begin{proof}
For weakly differentiable $u$ with finite norm, first replace it by
$\chi_Lu$. The tail terms tend to zero, and
\[
    \|u\,\d\chi_L\|_{L^p(\nu_a)}
       \le CL^{-1/2}\|u\|_{L^p(\nu_a)}\longrightarrow0.
\]
On each compact set, the density is smooth and positive, so ordinary
Sobolev smoothing gives compactly supported smooth approximations.

For $z_{k\ell}(s)=P_{a,s}(u_k-u_\ell)$, the estimates of
Lemma~\ref{lem:weighted-heat} for compactly supported smooth data give
\[
    \sup_{s\in[a,a']}\bigl(\|z_{k\ell}(s)\|_p
                   +\|\d z_{k\ell}(s)\|_p\bigr)
       \le\|u_k-u_\ell\|_{W^{1,p}(\nu_a)}.
\]
For $p=2$, \eqref{eq:hessian-energy} also controls
$\nabla^2z_{k\ell}$ in spacetime $L^2$. These estimates produce the
weighted limits. On relatively compact spacetime cylinders the heat density has a
positive lower bound, so the weighted bounds imply ordinary local
parabolic bounds. The limits are smooth there and solve the scalar heat
equation. At any fixed later time they have finite weighted Sobolev norm;
the cutoff approximation just proved shows that they lie in the same
Sobolev completion at that time.
\end{proof}

\begin{lemma}\label{lem:moments}
Consider flows induced by K\"ahler--Ricci shrinkers of complex
dimension $n$, based at time zero. For each compact interval
$I\Subset(-\infty,0)$ and each integer $k\ge1$,
\begin{equation}\label{eq:b-moments}
    \sup_{s\in I}\int_X|b(s)|^k\,\d\nu_s\le C(n,I,k).
\end{equation}
Furthermore,
\begin{equation}\label{eq:integral-input}
    \int_I\!\int_X
    \bigl(|\Ric|^2+|\nabla^2b|^2+|\d b|^4\bigr)
       \,\d\nu_s\,\d s\le C(n,I).
\end{equation}
For each fixed induced flow and compact time interval, $b$ has at most
quadratic spatial growth.
\end{lemma}

\begin{proof}
The exponential moment estimate in
\cite[Lemma~A.2]{liwang2024heat}, together with the heat-kernel upper
bound \cite[Theorem~4.4]{liwang2024heat}, gives \eqref{eq:b-moments}.
The estimates in \eqref{eq:integral-input} are
\cite[Lemmas~A.1 and~A.3]{liwang2024heat}, with the
backward-time factor bounded above and below on $I$.
The two-sided heat-kernel bounds
\cite[Theorems~4.9, 4.16 and Remark~4.17]{liwang2024heat} give the last
assertion. Its growth constant may depend on the fixed flow and base point.
\end{proof}

Let $(X,\omega,f)$ be a K\"ahler--Ricci shrinker satisfying
\eqref{eq:shrinker}. Let $\psi_t$, $t<0$, be the diffeomorphisms satisfying
\[
    \partial_t\psi_t=\frac{\nabla_g f\circ\psi_t}{2(-t)},
    \qquad \psi_{-1}=\Id.
\]
The canonical flow $\omega_t=(-t)\psi_t^*\omega$ satisfies
$\partial_t\omega_t=-\Ric(\omega_t)$. Set
\begin{equation}\label{eq:canonical-potential}
    F(x,t)=(-t)\psi_t^*\rho^2(x).
\end{equation}

We record the standard identities for $F$ in our normalization; see
\cite[Section~2]{liwang2020heat1} and
\cite[(2.6)--(2.9)]{liwang2024heat}. The integral identity
\eqref{eq:F-moment} follows by heat duality \eqref{eq:heat-duality}.

\begin{lemma}\label{lem:canonical}
On $X\times(-\infty,0)$,
\begin{equation}\label{eq:F-identities}
\begin{aligned}
    \partial_tF&=-(-t)\scal,&
    \Delta F&=n-(-t)\scal,\\
    \tfrac12|\d F|^2&=F-(-t)^2\scal,&
    \nabla^2F&=g(t)-(-t)\Ric.
\end{aligned}
\end{equation}
Consequently, $\heat F=-n$, $|\d F|^2\le2F$, and
\begin{equation}\label{eq:F-moment}
    \int_XF(\cdot,t_0-\tau)\,\d\nu_{t_0-\tau}
       =F(x_0,t_0)+n\tau.
\end{equation}
\end{lemma}

For a function $u$ whose Hessian $H=\nabla^2u$ is $J$-invariant
and tangent vectors $\xi,\eta$,
\begin{equation}\label{eq:Kahler-Hess}
    \ddc u(\xi,\eta)=-H(\xi,J\eta)+H(\eta,J\xi)=2H(J\xi,\eta).
\end{equation}
With the usual real two-form norm, this gives
$|\ddc u|^2=2|H|^2$. Since $\nabla^2F$ is $J$-invariant,
\eqref{eq:Kahler-Hess} applies to $F$.

\subsection{Weighted projection and the scalar curvature estimate}

Fix $x_0^*=(x_0,t_0)$ and $\tau>0$, and set $s=t_0-\tau$. Write
\begin{equation}\label{eq:G-definition}
    \cG_{x_0^*,\tau}
       =\overline{\{\d u:u\in C_c^\infty(X)\}}
                    ^{\,L^2(T^*X,g(s);\nu_s)}.
\end{equation}
Every element of this space is closed in the sense of distributions.
The Poincar\'e inequality \eqref{eq:Poincare} shows that each such element
has a weighted $W^{1,2}$ primitive after fixing its mean.

\begin{lemma}\label{lem:projection}
The orthogonal projection of $\dc F(s)$ onto $\cG_{x_0^*,\tau}$ is
$\d u_s$, where $u_s$ is a mean-zero weighted $W^{1,2}$ function. The
residual
\begin{equation}\label{eq:residual}
    \beta(s)=\dc F(s)-\d u_s
\end{equation}
is smooth and satisfies
\begin{equation}\label{eq:residual-coclosed}
    \delta_{b(s)}\beta(s)=0.
\end{equation}
For $s<t<t_0$, set
\begin{equation}\label{eq:residual-evolution}
    \beta(t)=\dc F(t)-\d P_{s,t}u_s.
\end{equation}
Then
\begin{equation}\label{eq:Hodge-evolution}
    \partial_t\beta=-\tfrac12\Delta_{\d,g(t)}\beta,
    \qquad \d\beta(t)=\ddc F(t),
\end{equation}
and, for every $s<t<t_0$,
\begin{equation}\label{eq:Hodge-energy}
    \|\beta(t)\|_2^2
       +\int_s^t\!\int_X|\nabla\beta|^2\,\d\nu_r\,\d r
       \le\|\beta(s)\|_2^2.
\end{equation}
The evolved gradient $\d P_{s,t}u_s$ belongs to $\cG_{x_0^*,t_0-t}$.
\end{lemma}

\begin{proof}
The bound $|\dc F|^2=|\d F|^2\le2F$ and \eqref{eq:F-moment} give
$\dc F(s)\in L^2(\nu_s)$. Let $\zeta$ be its orthogonal projection onto
the closed subspace $\cG_{x_0^*,\tau}$. By definition, there are
$\varphi_j\in C_c^\infty(X)$ such that $\d\varphi_j\to\zeta$ in
$L^2(\nu_s)$. Subtract their means and set
\[
    u_j=\varphi_j-\int_X\varphi_j\,\d\nu_s\in W^{1,2}(\nu_s).
\]
 The Poincar\'e inequality gives
\[
    \|u_j-u_k\|_2^2
       \le\tau\|\d\varphi_j-\d\varphi_k\|_2^2.
\]
Thus $u_j$ converges in $W^{1,2}(\nu_s)$ to a mean-zero function $u_s$
with $\d u_s=\zeta$. This primitive is unique by the same inequality.
Moreover,
\[
    \|\d u_s\|_2\le\|\dc F(s)\|_2,
    \qquad \|u_s\|_2\le\sqrt{\tau}\,\|\dc F(s)\|_2.
\]
Orthogonality gives the Euler equation
\[
    \int_X\langle\dc F(s)-\d u_s,\d\varphi\rangle\,\d\nu_s=0
       \qquad(\varphi\in C_c^\infty(X)).
\]
Equivalently, in the weak sense,
\[
    -2\Delta u_s+\langle\d b(s),\d u_s\rangle
       =\delta_{b(s)}\dc F(s).
\]
The coefficients and right side are smooth, and the operator is
uniformly elliptic on relatively compact subsets. Local elliptic
regularity makes $u_s$ and $\beta(s)$ smooth and proves
\eqref{eq:residual-coclosed}.

On functions, $\Delta_\d=-2\Delta$, and the K\"ahler identities give
$\Delta_\d\dc=\dc\Delta_\d$. Since $J$ is fixed along the flow and
$\heat F=-n$,
\[
    \partial_t(\dc F)=\dc(\Delta F-n)
       =-\tfrac12\Delta_\d(\dc F).
\]
Also, $\Delta_\d\d=\d\Delta_\d$, so $\d P_{s,t}u_s$ solves the same
Hodge heat equation for $t>s$. This proves the evolution in
\eqref{eq:Hodge-evolution}; its exterior derivative identity follows
from $\d^2=0$ in \eqref{eq:residual-evolution}.

For the norm identity, the Weitzenb\"ock formula gives
$\Delta_\d\beta=-\operatorname{tr}_{g(t)}\nabla^2\beta+\Ric^\sharp\beta$.
As $\partial_tg^{ij}=\Ric^{ij}$, we obtain
\[
\begin{aligned}
    \partial_t|\beta|^2
       &=\Ric(\beta^\sharp,\beta^\sharp)
          +2\langle\partial_t\beta,\beta\rangle
         =\langle\operatorname{tr}_{g(t)}\nabla^2\beta,\beta\rangle,\\
    \Delta|\beta|^2
       &=\langle\operatorname{tr}_{g(t)}\nabla^2\beta,\beta\rangle+|\nabla\beta|^2.
\end{aligned}
\]
Subtracting gives
\begin{equation}\label{eq:Hodge-Bochner}
    \heat|\beta|^2=-|\nabla\beta|^2.
\end{equation}

To integrate this identity, choose $u_k\in C_c^\infty(X)$ converging to
$u_s$ in $W^{1,2}(\nu_s)$, and for $r\in[s,t]$ set
\[
    \beta_k(r)=\dc F(r)-\d P_{s,r}u_k.
\]
Each $\beta_k$ solves the Hodge heat equation and has at most linear
spatial growth on this interval: $|\d F|$ has linear growth, and
$|\d P_{s,r}u_k|$ is bounded by the maximum principle.
The Gaussian heat-kernel bounds and the weighted $L^2$ bound for
$|\d b|$ therefore give the integrability required in
Lemma~\ref{lem:cutoff-duality} for $B=|\beta_k|^2$.
Integrating \eqref{eq:Hodge-Bochner} with compact cutoffs and applying
Fatou's lemma gives
\[
    \|\beta_k(t)\|_2^2
       +\int_s^t\!\int_X|\nabla\beta_k|^2\,\d\nu_r\,\d r
       \le\|\beta_k(s)\|_2^2.
\]
Since $\beta_k-\beta_\ell=-\d P_{s,r}(u_k-u_\ell)$,
\eqref{eq:gradient-contraction} and \eqref{eq:hessian-energy} imply
\[
    \sup_{r\in[s,t]}\|\beta_k(r)-\beta_\ell(r)\|_2^2
       +\int_s^t\!\int_X|\nabla(\beta_k-\beta_\ell)|^2\,\d\nu_r\,\d r
       \le2\|\d(u_k-u_\ell)\|_2^2.
\]
Consequently, $\beta_k\to\beta$ in $L^\infty([s,t];L^2(\nu_r))$
and $\nabla\beta_k\to\nabla\beta$ in spacetime $L^2$; the identification
of the limits follows from Lemma~\ref{lem:approximation} and local
parabolic regularity. The initial norms converge as well, so passage
to the limit proves \eqref{eq:Hodge-energy}.
Finally, $P_{s,t}u_s\in W^{1,2}(\nu_t)$ by
Lemma~\ref{lem:approximation}. Approximation at time $t$ by
$C_c^\infty(X)$ functions shows that its gradient belongs to
$\cG_{x_0^*,t_0-t}$.
\end{proof}

We now use the energy
\begin{equation}\label{eq:D-definition}
    \mathcal D_{x_0^*}(\tau)=\tau\inf_{\zeta\in\cG_{x_0^*,\tau}}
       \int_X|\dc F(\cdot,t_0-\tau)-\zeta|_{g(t_0-\tau)}^2
                    \,\d\nu_{t_0-\tau}.
\end{equation}
The factor $\tau$ in \eqref{eq:D-definition} can be absorbed into the
one-form norm by using the metric $g(t_0-\tau)/\tau$.
This metric and the conjugate heat kernel measure are invariant under
parabolic rescaling. Hence $\mathcal D$ is unchanged when $F$ is only
reparametrized in time, as in \eqref{eq:rescaled-flow}.

\begin{proposition}\label{prop:coarse}
For every $x_0^*=(x_0,t_0)$ and $\tau>0$,
\begin{equation}\label{eq:coarse-growth}
    \mathcal D_{x_0^*}(\tau)\le2\tau F(x_0,t_0)+2n\tau^2,
\end{equation}
and
\begin{equation}\label{eq:coarse-contraction}
    \mathcal D_{x_0^*}(\tau)\le\frac18\mathcal D_{x_0^*}(8\tau).
\end{equation}
\end{proposition}

\begin{proof}
The choice $\zeta=0$ is admissible. Thus \eqref{eq:coarse-growth} follows from
$|\dc F|^2=|\d F|^2\le2F$ and \eqref{eq:F-moment}.
For \eqref{eq:coarse-contraction}, evolve the minimizing one-form at
$t_0-8\tau$ to $t_0-\tau$. The evolved gradient is admissible and its
residual has nonincreasing weighted norm. Hence
\[
    \mathcal D_{x_0^*}(\tau)\le\tau\|\beta(t_0-\tau)\|_2^2
       \le\tau\|\beta(t_0-8\tau)\|_2^2
       =\frac18\mathcal D_{x_0^*}(8\tau).\qedhere
\]
\end{proof}

\begin{proposition}\label{prop:scalar-reduction}
There is a dimensional constant $C_n$ such that
\begin{equation}\label{eq:scalar-reduction}
    |t_0|\scal(x_0,t_0)
       \le C_n\left(1+\frac{\mathcal D_{x_0^*}(|t_0|)}{|t_0|^2}\right).
\end{equation}
\end{proposition}

\begin{proof}
Evolve the minimizing one-form from time $2t_0$.
Using $\d\beta=\ddc F$, \eqref{eq:Kahler-Hess}, and
$|\d\beta|^2\le2|\nabla\beta|^2$, the energy estimate
\eqref{eq:Hodge-energy} in Lemma~\ref{lem:projection} gives
\begin{equation}\label{eq:F-Hess-energy}
    \int_{2t_0}^{t_0}\!\int_X|\nabla^2F|^2\,\d\nu_s\,\d s
       \le C\|\beta(2t_0)\|_2^2
       =C|t_0|^{-1}\mathcal D_{x_0^*}(|t_0|).
\end{equation}

For $2t_0\le s\le t_0$, the Hessian identity in \eqref{eq:F-identities} therefore yields
\begin{equation}\label{eq:Ric-Hess}
    |\Ric_{g(s)}|^2
       \le C_n|t_0|^{-2}\bigl(1+|\nabla^2F|^2\bigr).
\end{equation}
Set $S(s)=\int_X\scal(\cdot,s)\,\d\nu_s$. By
\eqref{eq:scalar-average}, $S(2t_0)\le n/|t_0|$. The scalar curvature
evolution equation
\[
    (\partial_t-\Delta)\scal=\tfrac12|\Ric|^2
\]
and \eqref{eq:integrated-heat-duality} give
\[
    S(a')-S(2t_0)=\tfrac12\int_{2t_0}^{a'}\!\int_X|\Ric|^2\,\d\nu_s\,\d s
       \qquad(2t_0<a'<t_0).
\]
Combining \eqref{eq:scalar-average},
\eqref{eq:F-Hess-energy} and \eqref{eq:Ric-Hess}, then letting
$a'\uparrow t_0$, gives
\[
    \scal(x_0,t_0)\le C_n|t_0|^{-1}
                         +C_n|t_0|^{-3}\mathcal D_{x_0^*}(|t_0|).
\]
The bound $0\le\scal(\cdot,s)\le F(\cdot,s)/(-s)^2$,
the polynomial growth of $F$, and Gaussian tails give uniform
integrability on a fixed final time interval. 
This proves \eqref{eq:scalar-reduction}.
\end{proof}

\section{The scalar curvature bound}\label{sec:scalar-bound}

The proof adapts the strategy of Jian--Song \cite{JS26}.
By Proposition~\ref{prop:scalar-reduction}, it suffices to bound
$\mathcal D_{x_0^*}(1)$ by a dimensional constant for every $x_0^*=(x_0,-1)$.
We obtain this bound by combining analysis near almost self-similar
points, a vanishing lemma for one-forms, and iteration over scales.
The basepoints are arbitrary smooth spacetime points.

\subsection{Analysis near almost self-similar points}\label{subsec:almost-selfsimilar}

We now consider sequences whose pointed Nash entropy drop tends to zero.
We first describe the limiting metric soliton and then regularize the
heat-kernel potentials on the approximating flows.
We use the notion of metric soliton introduced by Bamler
\cite{bamler2023compactness}; see also
\cite[Definition~6.14]{liwang2024heat}.

Let $(X_i,g_i(s))_{s\in(-\infty,0]}$ be flows induced by
K\"ahler--Ricci shrinkers of complex dimension $n$, with smooth
basepoints $x_i^*=(x_i,0)$.
For $s<0$, let $\nu_{i,s}$ and $b_i(s)$ be the conjugate heat kernel
measures and heat-kernel potentials based at $x_i^*$, and write
$\NN_i(\tau)=\NN_{x_i^*}(\tau)$ and $\WW_i(\tau)=\WW_{x_i^*}(\tau)$
for $\tau>0$. We denote the scalar heat propagator of $g_i(s)$ by
$P^{(i)}_{a,a'}$.

By \eqref{eq:automatic-entropy} and
\cite[Theorems~6.10 and~6.12(1), (4)]{liwang2024heat}, a subsequence
has a limiting metric flow on compact time intervals, with regular
part $\RR$ and distinguished conjugate heat kernel measures $\nu_s$.
For every $s<0$, we have $\nu_s(\RR_s)=1$. Through time-preserving
embeddings, the metrics, time vector fields, and heat-kernel potentials
converge smoothly on compact subsets of $\RR$. Since the complex structures are parallel, they also
converge smoothly after passage to a further subsequence.

\begin{proposition}[Metric soliton limit]\label{prop:shrinking-limit}
For the sequence above, suppose that $\epsilon_i\downarrow0$ and
\begin{equation}\label{eq:entropy-flat}
    \NN_i(\epsilon_i)-\NN_i(\epsilon_i^{-1})\longrightarrow0.
\end{equation}
Then the limit is a metric soliton. On its regular part,
\begin{equation}\label{eq:limit-soliton}
    \Ric+\nabla^2b=\frac{g}{|s|},\qquad s\in(-\infty,0).
\end{equation}
Moreover, $\RR\cong\RR_{-1}\times(-\infty,0)$, with
$g(s)=|s|g(-1)$ and $\nu_s=\nu_{-1}$ under this identification.
\end{proposition}

\begin{proof}
Pass to a subsequence with $\NN_i(1)\to c\in[-A_n,0]$.
Monotonicity and \eqref{eq:entropy-flat} imply
$\NN_i(\tau)\to c$ for every fixed $\tau>0$.
Entropy continuity for these metric-flow limits
\cite[Theorems~2.10 and~2.14]{bamler3}, in the induced-shrinker setting
of \cite[Section~6]{liwang2024heat}, shows that the limiting Nash
entropy is constant. The rigidity theorem
\cite[Theorem~2.18]{bamler3}, applied in the real-flow convention $2g$,
then gives the asserted product decomposition and
\eqref{eq:limit-soliton}. In our normalization, the product time vector
field is $\partial_s-\tfrac12\nabla b$.
\end{proof}

\begin{lemma}\label{lem:small-defect}
In the setting of Proposition~\ref{prop:shrinking-limit}, at time $s<0$, set
\begin{equation}\label{eq:W-defect}
    W_i(u)=-s(\scal_i+2\Delta u-\tfrac12|\d u|^2)+u-2n-\NN_i(1).
\end{equation}
Thus $W_i(b_i)=w(b_i)-\NN_i(1)$, with $w$ in \eqref{eq:Perelman-w} evaluated using $g_i(s)$ and
$\tau=-s$.
For each $I\Subset(-\infty,0)$,
\begin{equation}\label{eq:W-L1}
    \sup_{s\in I}\|W_i(b_i)(s)\|_{L^1(\nu_{i,s})}\longrightarrow0,
\end{equation}
and
\begin{equation}\label{eq:W-L43}
    \|W_i(b_i)\|_{L^{4/3}(I\times X_i,\,\d\nu_{i,s}\d s)}\longrightarrow0.
\end{equation}
\end{lemma}

\begin{proof}
Apply \cite[Proposition~7.1]{bamler3} at scale $r=1$. The entropy bound
$\NN_i(1)\ge-A_n$ and \eqref{eq:entropy-flat} give almost selfsimilarity
with parameter tending to zero. Then
\cite[Proposition~7.3, (7.8)]{bamler3} with $\alpha=0$ gives
\eqref{eq:W-L1} on every fixed $I$. In the $2g$ convention, the
integrand in \cite[(7.8)]{bamler3} is $|W_i(b_i)|$.
These estimates extend to complete flows induced by shrinkers by the
exhaustion argument explained in \cite[Section~6]{liwang2024heat}, using
the cutoff and moment bounds recorded in
Lemmas~\ref{lem:cutoff-duality} and~\ref{lem:moments}.

Lemma~\ref{lem:moments} also bounds $W_i(b_i)$ in spacetime $L^2(I)$:
the terms $\Delta b_i$, $|\d b_i|^2$ and $\scal_i$ are controlled by
\eqref{eq:integral-input}. Since
$\|W_i(b_i)\|_{L^1(I)}\le |I|\sup_{s\in I}\|W_i(b_i)(s)\|_1\to0$, interpolation gives
\[
    \|W_i(b_i)\|_{L^{4/3}(I)}
       \le\|W_i(b_i)\|_{L^1(I)}^{1/2}\|W_i(b_i)\|_{L^2(I)}^{1/2}
       \longrightarrow0,
\]
which proves \eqref{eq:W-L43}.
\end{proof}

We next regularize the heat-kernel potentials of the sequence in
Proposition~\ref{prop:shrinking-limit}.
All spacetime integrals below use $\d\nu_{i,s}\,\d s$.
We suppress this measure in norm subscripts.
The construction adapts the heat regularization of Hallgren--Jian
\cite[Proposition~3.2]{HJ} and Fang-Li
\cite[Lemma~3.1]{FangLiSingular}.

\begin{proposition}\label{prop:strong-potential}
Assume the hypotheses of Proposition~\ref{prop:shrinking-limit}.
There are smooth functions $h_i$ on $X_i\times[-8,-1/2]$ such that
\begin{equation}\label{eq:strong-equation}
    (-s)\heat h_i=h_i-\NN_i(1)-n,
\end{equation}
\begin{equation}\label{eq:strong-gradient}
    \sup_{s\in[-8,-1/2]}
       \int_{X_i}|\d(h_i-b_i)|^2\,\d\nu_{i,s}\longrightarrow0,
\end{equation}
and
\begin{equation}\label{eq:strong-defect}
    \int_{-8}^{-1/2}\!\int_{X_i}
       \left|\Ric_{g_i}+\nabla^2h_i+\frac{g_i}{s}\right|^2
       \,\d\nu_{i,s}\,\d s\longrightarrow0.
\end{equation}
The construction also gives
\begin{equation}\label{eq:strong-L4}
    \sup_{s\in[-8,-1/2]}\|\d h_i(s)\|_{L^4(\nu_{i,s})}\le C(n).
\end{equation}
\end{proposition}

\begin{proof}
We work on the larger interval $[-40,-1/4]$. All constants below depend
only on $n$ and these fixed endpoints.

By Lemma~\ref{lem:moments}, choose $a_i\in[-40,-32]$ such that
\begin{equation}\label{eq:initial-regular-time}
    \int_{X_i}\bigl(|\d b_i(a_i)|^4
          +|\nabla^2b_i(a_i)|^2+|\Ric_{g_i(a_i)}|^2\bigr)
             \,\d\nu_{i,a_i}\le C.
\end{equation}
Define
\begin{equation}\label{eq:Q-u-e}
\begin{aligned}
    Q_i(s)&=-s\bigl(b_i(s)-\NN_i(1)\bigr),\\
    u_i(s)&=P^{(i)}_{a_i,s}Q_i(a_i)-n(s-a_i),\\
    h_i(s)&=-\frac{u_i(s)}{s}+\NN_i(1),
    \qquad e_i(s)=u_i(s)-Q_i(s).
\end{aligned}
\end{equation}
For each fixed $i$, the initial function has polynomial growth, and its
scalar heat evolution is obtained by the weighted approximation in
Lemmas~\ref{lem:weighted-heat} and~\ref{lem:approximation}.

The conjugate heat-kernel equation is
\[
    \partial_s b_i=-\Delta b_i+\tfrac12|\d b_i|^2-\scal_i-\frac ns.
\]
By \eqref{eq:W-defect}, it follows that
\begin{equation}\label{eq:e-equation}
    \heat Q_i=-n-W_i(b_i),\qquad
    \heat u_i=-n,\qquad
    \heat e_i=W_i(b_i),\qquad e_i(a_i)=0.
\end{equation}
The equation \eqref{eq:strong-equation} follows immediately.

Heat contraction and \eqref{eq:b-moments} give
\begin{equation}\label{eq:e-L4}
    \sup_{s\in[a_i,-1/4]}
       \bigl(\|u_i(s)\|_4+\|Q_i(s)\|_4+\|e_i(s)\|_4\bigr)\le C.
\end{equation}
Testing \eqref{eq:e-equation} against $e_i$ gives
\[
    \frac12\int e_i(a')^2\,\d\nu_{i,a'}
      +\tfrac12\int_{a_i}^{a'}\!\int|\d e_i|^2\,\d\nu_{i,s}\,\d s
       =\int_{a_i}^{a'}\!\int e_iW_i(b_i)\,\d\nu_{i,s}\,\d s.
\]
H\"older's inequality, \eqref{eq:e-L4} and \eqref{eq:W-L43} imply
\begin{equation}\label{eq:e-small}
    \sup_{s\in[a_i,-1/4]}\|e_i(s)\|_2^2
       +\int_{a_i}^{-1/4}\!\int|\d e_i|^2\,\d\nu_{i,s}\,\d s
       \longrightarrow0.
\end{equation}
By \eqref{eq:initial-regular-time}, gradient contraction and
\eqref{eq:hessian-energy},
\begin{equation}\label{eq:u-derivative}
    \sup_{s\in[a_i,-1/4]}\|\d u_i(s)\|_4\le C,
    \qquad
    \int_{a_i}^{-1/4}\!\int|\nabla^2u_i|^2\,\d\nu_{i,s}\,\d s\le C.
\end{equation}
This gives \eqref{eq:strong-L4}.

To estimate the error in the soliton equation, set
\begin{equation}\label{eq:Z-definition}
    Z_i=-sW_i(h_i)
       =s^2\scal_i-2s\Delta u_i-\tfrac12|\d u_i|^2+u_i+2ns,
\end{equation}
and write
\begin{equation}\label{eq:E-definition}
    E_i=\Ric_{g_i}+\nabla^2h_i+\frac{g_i}{s}.
\end{equation}
In our convention, the square identity used in \cite{HJ} reads
\begin{equation}\label{eq:Z-square}
    \heat Z_i=\tfrac12s^2|E_i|^2.
\end{equation}
Indeed, $\heat u_i=-n$ and the Bochner and scalar curvature identities yield
\begin{align*}
    \heat(s^2\scal_i)&=\tfrac12s^2|\Ric_i|^2+2s\scal_i,\\
    \heat(-2s\Delta u_i)&=-s\langle\Ric_i,\nabla^2u_i\rangle
                                   -2\Delta u_i,\\
    \heat(-\tfrac12|\d u_i|^2)&=\tfrac12|\nabla^2u_i|^2,\\
    \heat(u_i+2ns)&=n.
\end{align*}
Their sum is
\[
    \frac12\left|-s\Ric_i+\nabla^2u_i-g_i\right|^2,
\]
which proves \eqref{eq:Z-square}.

A second identity controls the weighted mean of $Z_i$:
\begin{equation}\label{eq:Z-mean}
    \int Z_i\,\d\nu_{i,s}
       =-s\int W_i(b_i)\,\d\nu_{i,s}
          +\int e_i\,\d\nu_{i,s}
          -\tfrac12\int|\d e_i|^2\,\d\nu_{i,s}.
\end{equation}
To verify it, substitute $h_i=b_i-e_i/s$ into
$-sW_i(h_i)$. The mixed terms have integral
\[
    -2s\int\bigl(\Delta e_i
                  -\tfrac12\langle\d b_i,\d e_i\rangle\bigr)\,\d\nu_{i,s}=0.
\]
This gives \eqref{eq:Z-mean} for almost every $s$.

By \eqref{eq:W-L1}, \eqref{eq:e-small} and \eqref{eq:Z-mean},
\begin{equation}\label{eq:Z-mean-small}
    \int_{a_i}^{-1/4}
       \left|\int Z_i\,\d\nu_{i,s}\right|\d s\longrightarrow0.
\end{equation}
At the initial time, $u_i(a_i)=Q_i(a_i)$, so
\begin{equation}\label{eq:Z-initial}
    Z_i(a_i)=(-a_i)W_i(b_i)(a_i),\qquad \|Z_i(a_i)\|_1\longrightarrow0.
\end{equation}

Choose a smooth nonnegative function $\theta$ of time, equal to one on
$[-40,-1/2]$ and vanishing at $-1/4$. Integrating
\eqref{eq:Z-square} against $\theta$ gives
\begin{align}\label{eq:time-cutoff-square}
    \tfrac12\int_{a_i}^{-1/4}\!\theta(s)s^2
          \int|E_i|^2\,\d\nu_{i,s}\,\d s
     =-\int Z_i(a_i)\,\d\nu_{i,a_i}
       -\int_{a_i}^{-1/4}\!\theta'(s)
             \int Z_i\,\d\nu_{i,s}\,\d s.
\end{align}
The right side tends to zero by
\eqref{eq:Z-mean-small}--\eqref{eq:Z-initial}. Hence
\begin{equation}\label{eq:E-small-long}
    \int_{a_i}^{-1/2}\!\int|E_i|^2\,\d\nu_{i,s}\,\d s
       \longrightarrow0.
\end{equation}
Kato's inequality applied to \eqref{eq:Z-square} now gives
\begin{equation}\label{eq:Z-L1-small}
    \sup_{s\in[a_i,-1/2]}\|Z_i(s)\|_1
       \le\|Z_i(a_i)\|_1
            +\tfrac12\int_{a_i}^{-1/2}\!s^2
                    \int|E_i|^2\,\d\nu_{i,s}\,\d s
       \longrightarrow0.
\end{equation}
Equation~\eqref{eq:E-small-long} proves \eqref{eq:strong-defect};
\eqref{eq:Z-L1-small} will be used for the gradient comparison.

Weighted integration
by parts gives
\begin{equation}\label{eq:gradient-comparison-identity}
    -\frac{s}{2}\int|\d(b_i-h_i)|^2\,\d\nu_{i,s}
       =\int\bigl(W_i(b_i)-W_i(h_i)-b_i+h_i\bigr)\,\d\nu_{i,s}.
\end{equation}
The right side is bounded in absolute value by
\[
    \|W_i(b_i)(s)\|_1+\frac{\|Z_i(s)\|_1+\|e_i(s)\|_1}{-s}.
\]
Equations \eqref{eq:W-L1}, \eqref{eq:e-small} and
\eqref{eq:Z-L1-small} therefore prove
\eqref{eq:strong-gradient}, initially for almost every time.
Local smoothness and Fatou's lemma show that the left side is lower
semicontinuous in time. The same bound consequently holds at every time
in $[-8,-1/2]$, including both endpoints.

To apply Lemma~\ref{lem:cutoff-duality}, observe that
\eqref{eq:integral-input}, \eqref{eq:e-L4} and
\eqref{eq:u-derivative} give
\[
    Z_i\in L^2(\d\nu_{i,s}\,\d s),\quad
    \d b_i\in L^2(\d\nu_{i,s}\,\d s),\quad
    e_i^2\in L^2(\d\nu_{i,s}\,\d s).
\]
Thus $Z_i|\d b_i|$ and $e_i^2|\d b_i|$ are in $L^1$.
The source $s^2|E_i|^2$ is integrable,
since both $\Ric_i$ and $\nabla^2u_i$ have finite spacetime $L^2$ norm.
Lemma~\ref{lem:cutoff-duality} therefore applies to $Z_i$, $|Z_i|$ and
$e_i^2$. At the initial time, the identities are first applied with a
compactly supported cutoff and then passed to the limit as the cutoffs
exhaust the manifold.

For the spatial integration by parts, $\Delta e_i$ is integrable at
almost every time by the Hessian bounds, and $\d e_i,\d b_i$ are in
$L^2$. The cutoff error is bounded by
$\|\d e_i\|_2\|\d\chi_L\|_2$ and tends to zero.
All exhaustion limits are taken for a fixed $i$ before passing to the
sequence limit. This completes the proof.
\end{proof}

\subsection{The vanishing lemma}\label{subsec:vanishing}

The following identity describes the evolution of the weighted
divergence under Hodge heat flow. It is the real one-form version of
\cite[Lemma~3.1(ii)]{HallgrenAlgebraic}; see also
\cite[Lemma~6.1]{JS26}.

\begin{lemma}\label{lem:divergence}
Let $g(s)$, $s<0$, be a smooth real Ricci flow in arbitrary dimension
with $\partial_sg=-\Ric_g$, and set
$\heat=\partial_s-\tfrac12\Delta_{g(s)}$.
Let $\beta$ solve $\partial_s\beta=-\tfrac12\Delta_{\d,g(s)}\beta$.
Suppose that
$(-s)\heat h=h-c$ for a constant $c$. Then
\begin{equation}\label{eq:divergence-identity}
    \heat((-s)\delta_h\beta)
       =s\left\langle
           \Ric+\nabla^2h+\frac{g}{s},\nabla\beta
         \right\rangle.
\end{equation}
\end{lemma}

We now apply this identity to the one-forms arising from the weighted
projection.

\begin{proposition}\label{prop:vanishing}
Assume the setting of Proposition~\ref{prop:shrinking-limit}. Let $\beta_i$
be one-forms obtained from Lemma~\ref{lem:projection} by multiplication
by constants and parabolic rescaling. Suppose they are defined on
$[-8,-1/2]$ and satisfy
\begin{equation}\label{eq:beta-initial}
    \|\beta_i(-8)\|_2=1,\qquad
    (\delta_{b_i}\beta_i)(-8)=0.
\end{equation}
Suppose that they converge smoothly on compact subsets of the regular
spacetime for $-8<s<-1/2$, and that their limit $\beta$ is closed at time
$-1$. Then
\begin{equation}\label{eq:vanishing-conclusion}
    \|\beta_i(-1)\|_{L^2(\nu_{i,-1})}\longrightarrow0.
\end{equation}
\end{proposition}

\begin{proof}
Applying \eqref{eq:Hodge-energy} from Lemma~\ref{lem:projection}
to $\beta_i$ and using the normalization \eqref{eq:beta-initial}, we obtain
\begin{equation}\label{eq:beta-energy}
    \sup_{s\in[-8,-1/2]}\|\beta_i(s)\|_2\le1,
    \qquad
    \int_{-8}^{-1/2}\!\int|\nabla\beta_i|^2\,\d\nu_{i,s}\,\d s\le1.
\end{equation}
Let $h_i$ be the potentials in Proposition~\ref{prop:strong-potential}
and let $E_i$ be given by \eqref{eq:E-definition}. Put
$V_i=(-s)\delta_{h_i}\beta_i$. At time $-8$,
\begin{equation}\label{eq:div-initial-small}
    \|V_i(-8)\|_1
       \le8\|\d(h_i-b_i)(-8)\|_2\|\beta_i(-8)\|_2\longrightarrow0.
\end{equation}
By Lemma~\ref{lem:divergence} and \eqref{eq:strong-equation},
\[
    \heat V_i=s\langle E_i,\nabla\beta_i\rangle.
\]
By H\"older's inequality, the bounds \eqref{eq:integral-input},
\eqref{eq:strong-L4}, and \eqref{eq:beta-energy} imply
$V_i,\,|V_i||\d b_i|\in L^1(\d\nu_{i,s}\,\d s)$ on $[-8,-1]$.
Thus we may apply the Kato estimate \eqref{eq:L1-Kato} from
Lemma~\ref{lem:cutoff-duality},
and then Cauchy--Schwarz in spacetime to obtain
\begin{equation}\label{eq:div-h-small}
    \|\delta_{h_i}\beta_i(-1)\|_1 \le\|V_i(-8)\|_1 +8\left(\int_{-8}^{-1}\!\int|E_i|^2 \,\d\nu_{i,s}\,\d s\right)^{1/2} \left(\int_{-8}^{-1}\!\int|\nabla\beta_i|^2 \,\d\nu_{i,s}\,\d s\right)^{1/2} \longrightarrow0.
\end{equation}
Here $V_i(-1)=\delta_{h_i}\beta_i(-1)$, and the convergence follows from
\eqref{eq:div-initial-small}, \eqref{eq:strong-defect}, and
\eqref{eq:beta-energy}.

At time $-1$, replacing $h_i$ by $b_i$ costs at most
$\|\d(h_i-b_i)(-1)\|_2\|\beta_i(-1)\|_2=o(1)$. Consequently,
\begin{equation}\label{eq:div-b-small}
    \|\delta_{b_i}\beta_i(-1)\|_1\longrightarrow0.
\end{equation}
Smooth convergence on compact subsets of $\RR_{-1}$ and the closedness of the limit $\beta$ imply
\begin{equation}\label{eq:limit-harmonic}
    \d\beta=0,\qquad\delta_b\beta=0
       \qquad\text{on }\RR_{-1}.
\end{equation}

For each fixed $i$, approximate the initial scalar primitive defining
$\beta_i$ by compactly supported smooth functions, as in the proof of
Lemma~\ref{lem:projection}, and denote the resulting one-forms by
$\beta_{i,k}$. They have polynomial spatial growth on $[-8,-1]$.
The Hodge Bochner identity and the chain rule make
$(|\beta_{i,k}|^2+\epsilon^2)^{1/2}$ a nonnegative heat subsolution.
Apply \eqref{eq:hypercontractivity} and let $\epsilon\to0$. Since
$8\ge(8-1)/(2-1)$, this gives
\[
    \|\beta_{i,k}(-1)\|_{L^8(\nu_{i,-1})}
       \le\|\beta_{i,k}(-8)\|_{L^2(\nu_{i,-8})}.
\]
Lemma~\ref{lem:approximation} gives $L^2$ convergence of the one-forms
at both endpoints. Along a subsequence there is also almost-everywhere
convergence at time $-1$, so Fatou's lemma and convergence of the
initial norms yield
\begin{equation}\label{eq:beta-L8}
    \|\beta_i(-1)\|_{L^8(\nu_{i,-1})}
       \le\|\beta_i(-8)\|_{L^2(\nu_{i,-8})}=1.
\end{equation}
The limit therefore lies in $L^8(\nu_{-1})$, and hence in
$L^2(\nu_{-1})$.

At time $-1$, \eqref{eq:limit-soliton} gives
$\Ric+\nabla^2b=g$. Thus \eqref{eq:limit-harmonic} and the weighted
Weitzenb\"ock formula yield
\begin{equation}\label{eq:weighted-Weitzenbock}
    0=(\d\delta_b+\delta_b\d)\beta =\nabla_b^*\nabla\beta+(\Ric+\nabla^2b)^\sharp\beta =\nabla_b^*\nabla\beta+\beta.
\end{equation}
Here $\nabla_b^*$ is the adjoint of $\nabla$ in $L^2(\nu_{-1})$.
For a compactly supported cutoff $\chi$ in the regular slice, testing
against $\chi^2\beta$ and using Young's inequality gives
\begin{equation}\label{eq:Bochner-cutoff}
    \frac12\int\chi^2\bigl(|\nabla\beta|^2+|\beta|^2\bigr)\,\d\nu_{-1}
       \le2\int|\d\chi|^2|\beta|^2\,\d\nu_{-1}.
\end{equation}
Inequality \eqref{eq:Bochner-cutoff} also holds for compactly supported
Lipschitz functions on $\RR_{-1}$, by smooth approximation in $W^{1,2}$
on a relatively compact neighborhood of the support.
Fix $z\in\RR_{-1}$. For $L\ge1$, choose a Lipschitz cutoff
$0\le\rho_L\le1$, equal to one on $B(z,L)$ and vanishing outside
$B(z,2L)$, with $|\d\rho_L|\le C/L$ almost everywhere.
Bamler's codimension-four estimate
\cite[Theorems~2.16(d), 2.18]{bamler3} and cutoff construction
\cite[Lemma~15.27]{bamler3} apply through the extension of Li--Wang
\cite[Section~6, in particular Theorem~6.12]{liwang2024heat}.
On the metric soliton from Proposition~\ref{prop:shrinking-limit},
these give singular-set cutoffs with locally vanishing $L^p$
gradient norm on each bounded slice region for every $1\le p<4$.
See also \cite[Proposition~8.23]{FangLiStructure} and
\cite[Lemma~3.3]{HallgrenZhang} for related cutoff formulations
on singular shrinkers.
In particular, there are smooth functions $0\le\chi_\epsilon\le1$,
vanishing near the singular set and converging locally to one on
$\RR_{-1}$, such that $\rho_L\chi_\epsilon$ has compact support
in $\RR_{-1}$. Choosing $p=8/3$ and using the bounded heat density
from \cite[Theorems~4.4, 6.12(4)]{liwang2024heat} converts the
unweighted estimate to the weighted bound, for each fixed $L$,
\[
    \|\d\chi_\epsilon\|_{L^{8/3}(B(z,2L)\cap\RR_{-1},\,\nu_{-1})}
       \longrightarrow0.
\]
Taking $\chi=\rho_L\chi_\epsilon$ in \eqref{eq:Bochner-cutoff} and
using H\"older's inequality, we obtain
\[
    \int\rho_L^2\chi_\epsilon^2
       \bigl(|\nabla\beta|^2+|\beta|^2\bigr)\,\d\nu_{-1}
       \le \frac{C}{L^2}\|\beta\|_2^2
       +C\|\beta\|_8^2
       \|\d\chi_\epsilon\|_{L^{8/3}(B(z,2L)\cap\RR_{-1},\,\nu_{-1})}^2.
\]
Letting $\epsilon\to0$ and applying Fatou's lemma gives
\[
    \int_{B(z,L)\cap\RR_{-1}}
       \bigl(|\nabla\beta|^2+|\beta|^2\bigr)\,\d\nu_{-1}
       \le\frac{C}{L^2}\|\beta\|_2^2.
\]
Letting $L\to\infty$, we conclude that $\beta=0$.

We now prove global weighted $L^2$ convergence.
Choose compact $E\Subset\RR_{-1}$ with
$\nu_{-1}(E)>1-\delta$, and let $E_i$ be its image under the convergence
maps. For large $i$, $\nu_{i,-1}(X_i\setminus E_i)<2\delta$. By
\eqref{eq:beta-L8},
\[
    \int_{X_i\setminus E_i}|\beta_i(-1)|^2\,\d\nu_{i,-1}
       \le(2\delta)^{3/4}.
\]
The integral on $E_i$ tends to zero by smooth convergence. Letting
$\delta\to0$ proves \eqref{eq:vanishing-conclusion}.
\end{proof}

\subsection{Energy improvement and iteration}\label{subsec:energy-iteration}

We work at time $-1$ of the canonical ancient flow, where
$\omega_{-1}=\omega$. We first improve the estimate for $\mathcal D_{x_0^*}$ at scales
with small entropy drop, and then iterate to prove Theorem~\ref{thm:main}.

\begin{proposition}\label{prop:improved-contraction}
Fix an integer $n\ge1$ and $\epsilon>0$. There are constants
$\eta=\eta(n,\epsilon)\in(0,1/8)$ and
$C_\epsilon=C(n,\epsilon)<\infty$ with the following property.
For every K\"ahler--Ricci shrinker, every
$x_0^*=(x_0,-1)$, and every $\tau\ge1$ satisfying
\begin{equation}\label{eq:good-scale}
    \NN_{x_0^*}(\eta\tau)-\NN_{x_0^*}(\eta^{-1}\tau)\le\eta,
\end{equation}
we have
\begin{equation}\label{eq:improved-contraction}
    \mathcal D_{x_0^*}(\tau)\le\epsilon \mathcal D_{x_0^*}(8\tau)+C_\epsilon\tau^2.
\end{equation}
\end{proposition}

\begin{proof}
The case $\epsilon\ge1/8$ follows from
\eqref{eq:coarse-contraction}. Suppose that the result fails for some
$0<\epsilon<1/8$. Choose $\eta_i\downarrow0$ with $\eta_i<1/8$.
The negation of the proposition gives a sequence of shrinkers,
points $x_i^*=(x_i,-1)$ and scales $\tau_i\ge1$ such that
\begin{equation}\label{eq:contradiction-entropy}
    \NN_{x_i^*}(\eta_i\tau_i)-\NN_{x_i^*}(\eta_i^{-1}\tau_i)\le\eta_i,
\end{equation}
but
\begin{equation}\label{eq:contradiction-D}
    \mathcal D_{x_i^*}(\tau_i)>\epsilon \mathcal D_{x_i^*}(8\tau_i)+i\tau_i^2.
\end{equation}
Combining \eqref{eq:contradiction-D} with the coarse contraction \eqref{eq:coarse-contraction} gives
\begin{equation}\label{eq:scale-amplitude}
    \left(\frac18-\epsilon\right)\mathcal D_{x_i^*}(8\tau_i)>i\tau_i^2,
    \qquad
    \frac{1+\tau_i}{\sqrt{\mathcal D_{x_i^*}(8\tau_i)}}\longrightarrow0.
\end{equation}

Let $g^{(i)}(t)$ be the canonical ancient Ricci flow of the $i$th shrinker,
and let $F^{(i)}$ be its potential defined by \eqref{eq:canonical-potential}.
Let
\begin{equation}\label{eq:rescaled-flow}
    g_i(s)=\tau_i^{-1}g^{(i)}(-1+\tau_i s),\qquad
    F_i(s)=F^{(i)}(-1+\tau_i s).
\end{equation}
Since $t=-1$ corresponds to $s=0$, the rescaled flows are based at
$(x_i,0)$, as in Proposition~\ref{prop:shrinking-limit}, with
$\epsilon_i=\eta_i$.
Since $\partial_sF_i-\tfrac12\Delta_{g_i(s)}F_i=-n\tau_i$,
$\dc F_i$ again solves the Hodge heat equation.
At $s=-8$, let $\d\varphi_i$ be the weighted projection of
$\dc F_i(-8)$ given by Lemma~\ref{lem:projection}. The residual has
squared norm $\mathcal D_{x_i^*}(8\tau_i)/8$ with respect to $g_i(-8)$. Set
\begin{equation}\label{eq:rescaled-beta}
    \beta_i(s)=\sqrt{\frac{8}{\mathcal D_{x_i^*}(8\tau_i)}}
       \bigl(\dc F_i(s)-\d P^{(i)}_{-8,s}\varphi_i\bigr).
\end{equation}
Then
\begin{equation}\label{eq:normalized-decomposition}
    \|\beta_i(-8)\|_2=1,
    \qquad(\delta_{b_i}\beta_i)(-8)=0,
\end{equation}
and these one-forms satisfy \eqref{eq:beta-energy}.

By \eqref{eq:contradiction-entropy} and
Proposition~\ref{prop:shrinking-limit}, a subsequence of the pointed flows
converges to a metric soliton. We use the time-preserving embeddings
on its regular spacetime from Subsection~\ref{subsec:almost-selfsimilar}.
Let $Q'\Subset Q\Subset\RR$ be relatively compact spacetime cylinders. The heat density has a positive lower bound
on $Q$. Thus
\eqref{eq:beta-energy} gives a uniform unweighted spacetime $L^2$ bound
for $\beta_i$ on $Q$. The Hodge equation has smoothly converging
coefficients there. Interior parabolic estimates give uniform $C^k$
bounds on $Q'$ for every $k$. A diagonal subsequence converges smoothly
on the regular spacetime for $-8<s<-1/2$.

At time $-1$, the canonical Hessian identity becomes
\begin{equation}\label{eq:rescaled-Hess}
    \nabla^2F_i(-1)
       =\tau_i g_i(-1)
          -(1+\tau_i)\Ric_{g_i(-1)}.
\end{equation}
On each compact subset of $\RR_{-1}$, the rescaled metrics and Ricci tensors
are uniformly bounded. By \eqref{eq:scale-amplitude} and
\eqref{eq:rescaled-Hess}, the Hessian of $F_i(-1)$ divided by
$\sqrt{\mathcal D_{x_i^*}(8\tau_i)}$ tends to zero there. By
\eqref{eq:rescaled-beta}, $\d^2=0$, and \eqref{eq:Kahler-Hess},
\begin{equation}\label{eq:limit-closed}
    \d\beta_i(-1)=\sqrt{\frac{8}{\mathcal D_{x_i^*}(8\tau_i)}}
       \,\ddc F_i(-1)\longrightarrow0.
\end{equation}
The limiting one-form is closed. Proposition~\ref{prop:vanishing}
therefore gives $\|\beta_i(-1)\|_2\to0$.

By Lemma~\ref{lem:projection}, $\d P^{(i)}_{-8,-1}\varphi_i$ belongs to
$\cG_{x_i^*,\tau_i}$ under the rescaling \eqref{eq:rescaled-flow}.
Then using \eqref{eq:rescaled-beta}, we obtain
\[
    \frac{8\mathcal D_{x_i^*}(\tau_i)}{\mathcal D_{x_i^*}(8\tau_i)}
       \le\|\beta_i(-1)\|_2^2\longrightarrow0.
\]
This contradicts \eqref{eq:contradiction-D}, which bounds the same ratio
below by $8\epsilon$.
\end{proof}

We complete the proof by iterating the contraction over geometric scales.

\begin{proof}[Proof of Theorem~\ref{thm:main}]
Fix $\epsilon=1/128$ and choose $\eta,C_\epsilon$ from
Proposition~\ref{prop:improved-contraction}. Fix $x_0^*=(x_0,-1)$ and set
\[
    \tau_j=8^j,\qquad j=0,1,2,\ldots.
\]
Let $\mathcal B$ be the set of indices for which
\eqref{eq:good-scale} fails at $\tau_j$.

The intervals
\[
    I_j=[\log\tau_j-\log(\eta^{-1}),\log\tau_j+\log(\eta^{-1})]
\]
have overlap multiplicity at most
$L_\eta=\lfloor2\log(\eta^{-1})/\log8\rfloor+2$. The nonincreasing
function $u\mapsto\NN_{x_0^*}(e^u)$ has total decrease at most $A_n$, by
\eqref{eq:entropy-bounds}. Hence
\begin{equation}\label{eq:bad-count}
    \#\mathcal B\le B,
    \qquad B=\left\lceil L_\eta A_n/\eta\right\rceil.
\end{equation}

Use \eqref{eq:improved-contraction} when $j\notin\mathcal B$, and
\eqref{eq:coarse-contraction} otherwise. Thus
\begin{equation}\label{eq:recursion}
    \mathcal D_{x_0^*}(\tau_j)\le a_j\mathcal D_{x_0^*}(\tau_{j+1})+C_\epsilon\tau_j^2,
    \qquad
    a_j=\begin{cases}
       \epsilon,&j\notin\mathcal B,\\
       1/8,&j\in\mathcal B.
    \end{cases}
\end{equation}
By \eqref{eq:bad-count}, for any $j\geq 1$,
\[
    \prod_{k=0}^{j-1}a_k\le(8\epsilon)^{-B}\epsilon^j
       \le C(n)\epsilon^j.
\]
Iterating \eqref{eq:recursion} gives
\begin{equation}\label{eq:iteration-final}
    \mathcal D_{x_0^*}(1)
       \le C(n)\epsilon^m\mathcal D_{x_0^*}(8^m)
          +C(n)\sum_{j=0}^{m-1}(64\epsilon)^j.
\end{equation}
The sum is uniformly bounded. By \eqref{eq:coarse-growth},
\begin{equation}\label{eq:terminal-vanishes}
    \epsilon^m\mathcal D_{x_0^*}(8^m)
       \le2(8\epsilon)^mF(x_0,-1)
                      +2n(64\epsilon)^m\longrightarrow0.
\end{equation}
Here $x_0$ is fixed while $m\to\infty$.
We obtain
\begin{equation}\label{eq:uniform-D}
    \mathcal D_{x_0^*}(1)\le C(n)
\end{equation}
with a constant independent of $x_0$. Then
Proposition~\ref{prop:scalar-reduction} at $t_0=-1$ gives
$\scal_\omega(x_0)\le C(n)$ for every $x_0\in X$, proving
\eqref{eq:main-bound}.
\end{proof}

\section{The time-zero slice and the Fano fibration}\label{sec:time-zero-fibration}

Throughout this section, let
$(X^n,\omega,f)$ be a noncompact K\"ahler--Ricci
shrinker, and let $\pi:X\to Y$ be its polarized Fano fibration
\cite[Theorem~3.1 and Definitions~2.1, 2.5]{SunZhang}. We use the canonical flow $\omega_t=(-t)\psi_t^*\omega$
and the potential $F$ from \eqref{eq:canonical-potential}, and write
$d_t$ for the distance of the associated Riemannian metric $g_t$.
Fix a minimum point $p$ of $f$ and write $0=\pi(p)$.
By Theorem~\ref{thm:main}, $0\leq\scal_\omega\leq C_n$. Unless otherwise specified, constants in this section
may depend on the fixed shrinker.

\subsection{The canonical slice and local distance control}\label{subsec:time-zero-fibration}
We apply the spacetime construction of \cite[Section~5]{LZ26} to
$2g_t$, which satisfies the Ricci flow equation in the normalization
used there, and divide the resulting distances by $\sqrt2$.
Explicitly, for $x^*=(x,t)$ and $y^*=(y,s)$ with $s\leq t<0$, set
\[
    d^*(x^*,y^*)
       =\inf_{a\leq s}\max\left\{
          \sqrt{\frac{t-a}{2}},\,
          d_{W_1}^{g(a)}(\nu_{x,t;a},\nu_{y,s;a})\right\},
\]
where $d_{W_1}^{g(a)}$ is the $1$-Wasserstein distance on $(X,d_a)$ and
$\nu_{x,t;a}$ is the conjugate heat kernel measure, with
$\nu_{x,t;t}=\delta_x$.
Let $(Z,d_Z)$ be the metric completion of
$(X\times(-\infty,0),d^*)$. The time function extends continuously
to $\mathfrak t:Z\to(-\infty,0]$, and the added points form
$Z_0=\mathfrak t^{-1}(0)$. Conjugate heat kernel measures extend to
points of $Z_0$, and its canonical slice distance is
\[
    d_0^Z(z,z')
       =\lim_{a\uparrow0}d_{W_1}^{g(a)}(\nu_{z;a},\nu_{z';a}),
       \qquad z,z'\in Z_0.
\]
The distance $d_Z$ induces the spacetime topology on $Z_0$.
By \cite[Lemmas~5.3--5.4]{LZ26}, limits along worldlines define a
surjective map
\begin{equation}\label{eq:time-zero-map}
    \Phi:X\longrightarrow Z_0,
    \qquad \Phi(x)=\lim_{t\uparrow0}(x,t),
\end{equation}
where the limit is taken in the spacetime completion, and
\begin{equation}\label{eq:time-zero-distance}
    d_0^Z(\Phi(x),\Phi(x'))=\lim_{t\uparrow0}d_t(x,x').
\end{equation}
Put $\bar p=\Phi(p)$. The continuous comparison map
$\Phi_0:Z_0\to Y$ of \cite[Lemma~6.1]{LZ26} satisfies
$\pi=\Phi_0\circ\Phi$, where $Z_0$ initially carries the spacetime
topology. We will prove that $\Phi_0$ is a homeomorphism also for the
topology of $d_0^Z$.

The canonical identities give
\(   \partial_\tau F(x,-\tau)
       =\scal_\omega(\psi_{-\tau}(x))\in[0,C_n].
\)
Consequently, $F(\cdot,-\tau)$ converges locally uniformly to a
continuous function $F_0$, and, for $0<\tau\leq1$,
\begin{equation}\label{eq:tangent-potential-estimates}
\begin{split}
    0&\leq F(x,-\tau)-F_0(x)\leq C_n\tau,\\
    \rho^2(x)-C_n&\leq F_0(x)\leq F(x,-\tau)\leq\rho^2(x).
\end{split}
\end{equation}In particular, $F_0$ is proper on $X$.
The potential--distance comparison \cite[(5.2)]{LZ26}, with the
normalization above, gives
\begin{equation}\label{eq:time-zero-radius}
    d_0^Z(\bar p,\Phi(x))=\sqrt{2F_0(x)}.
\end{equation}

\begin{lemma}\label{lem:polynomial-metric-comparison}
For every compact set $Q\Subset X$, there are constants $C_Q<\infty$
and $N_Q\geq1$ such that
\[
    \omega_{-\tau}\leq C_Q\tau^{-N_Q}\omega
    \qquad\text{on }Q,\quad 0<\tau\leq1.
\]
\end{lemma}

\begin{proof}
Choose an equivariant embedding
$(h,\Psi):X\hookrightarrow\mathbb C^L\times\mathbb P^m$ as in
\cite[Section~5.2]{ZhangUniform}, where $h$ consists of homogeneous
generators of the affine coordinate ring of $Y$, pulled back by $\pi$,
and $\Psi$ is defined by basepoint-free eigensections of a power of
$-K_X$. Write
\[
    \omega_{\rm aff}=\sqrt{-1}\sum_{j=1}^L
       \partial h_j\wedge\bar\partial\overline{h_j},
    \qquad \omega_{\rm pr}=\Psi^*\omega_{\rm FS},
    \qquad \eta=\omega_{\rm aff}+\omega_{\rm pr}.
\]By \cite[Lemma~5.4]{ZhangUniform},
\[
    \operatorname{tr}_{\omega}\omega_{\rm pr}
       \leq C(1+\rho^2)^A.
\]
The vector field $\tfrac12\nabla f$ induces a smooth vector field on
$\mathbb P^m$. Its Lie derivative of $\omega_{\rm FS}$ is bounded by
$c\omega_{\rm FS}$ in absolute value. Equivariance therefore gives
$\psi_{-\tau}^*\omega_{\rm pr}\geq\tau^c\omega_{\rm pr}$ for
$0<\tau\leq1$. Hence
\[
    \operatorname{tr}_{\omega_{-\tau}}\omega_{\rm pr} \leq\tau^{-1-c} (\operatorname{tr}_\omega\omega_{\rm pr})\circ\psi_{-\tau} \leq C\tau^{-1-c} \bigl(1+\tau^{-1}F(\cdot,-\tau)\bigr)^A \leq C_Q\tau^{-B}\qquad\text{on }Q,
\]
where $B=1+c+A$ and $c,A\geq0$. By
\cite[Proposition~4.4]{LZ26},
$\omega_{\rm aff}\leq C_Q\omega_{-\tau}$ on $Q$. Thus
$\omega_{-\tau}\geq C_Q^{-1}\tau^B\eta$ there.
Since $\scal_{\omega_t}\geq0$, volume evolution gives
$\omega_{-\tau}^n\leq\omega^n$. The eigenvalues of $\omega_{-\tau}$
relative to $\eta$ are therefore bounded below by
$C_Q^{-1}\tau^B$ and have uniformly bounded product. Each is at most
$C_Q\tau^{-B(n-1)}$. Comparing $\eta$ and $\omega$ on $Q$ proves the
claim with $N_Q=\max\{1,B(n-1)\}$.
\end{proof}

\begin{lemma}\label{lem:tangent-distance-modulus}
On every relatively compact coordinate ball, after restricting to a
smaller ball, there are constants $C<\infty$ and $\alpha>0$ such that
\begin{equation}\label{eq:tangent-distance-modulus}
    d_{-\tau}(x,x')\leq C|x-x'|^\alpha,
    \qquad 0<\tau\leq1.
\end{equation}
\end{lemma}

\begin{proof}
Fix a slightly larger compact coordinate neighborhood $Q$.
Lemma~\ref{lem:polynomial-metric-comparison}, the inequality
$|\d F|_{g_{-\tau}}^2\leq2F$, and
\eqref{eq:tangent-potential-estimates} give, for some $N\geq1$,
\begin{equation}\label{eq:tangent-metric-growth}
    \omega_{-\tau}\leq C_Q\tau^{-N}\omega,
    \qquad |\d F(\cdot,-\tau)|_{\rm Euc}
       \leq C_Q\tau^{-N/2}\quad\text{on }Q.
\end{equation}
Set $\alpha=1/(N+2)$ and $\delta=|x-x'|$. For $\delta$ sufficiently
small, if $\tau\geq\delta^{2\alpha}$, integrating the gradient
estimate along the coordinate segment gives
\[
    |F(x,-\tau)-F(x',-\tau)|
       \leq C_Q\delta\tau^{-N/2}
       \leq C_Q\delta^{2\alpha}.
\]
If $\tau<\delta^{2\alpha}$, compare both values with time
$-\delta^{2\alpha}$ using $|\partial_\tau F|\leq C_n$ and then
apply the preceding estimate. Thus
\begin{equation}\label{eq:tangent-potential-modulus}
    |F(x,-\tau)-F(x',-\tau)|\leq C_Q\delta^{2\alpha},
    \qquad0<\tau\leq1.
\end{equation}
The soliton equation yields
$\partial_\tau(\tau^{-1}\omega_{-\tau})
=-\tau^{-2}\sqrt{-1}\partial\bar\partial F(\cdot,-\tau)$.
Integrating gives
\begin{equation}\label{eq:tangent-kahler-potential}
    \omega_{-\tau}
      =\tau\omega+\sqrt{-1}\partial\bar\partial
        \left(\tau\int_\tau^1\frac{F(\cdot,-u)}{u^2}\,\d u\right).
\end{equation}
Since $\tau\int_\tau^1u^{-2}\,\d u=1-\tau$, these forms have local
potentials $\varphi_\tau$ with uniformly bounded $C^{2\alpha}$ norms by \eqref{eq:tangent-potential-modulus}.
The local distance estimate of Guedj--Guenancia--Zeriahi
\cite[Proposition~1.4]{GGZ} applies with modulus $Cr^{2\alpha}$.
We give the integration-by-parts argument of
\cite[Theorem~4.1]{YangLi}.
For a fixed point $x$ in the smaller ball, put
$u=d_{-\tau}(x,\cdot)$. Since $|\d u|_{g_{-\tau}}\leq1$
almost everywhere,
$|\d u|_{\rm Euc}^2\leq C\operatorname{tr}_{\rm Euc}\omega_{-\tau}$.
For every Euclidean ball $B(z,2r)\Subset Q$, choose a smooth cutoff
$0\leq\chi\leq1$, equal to one on $B(z,r)$, supported in $B(z,2r)$,
with $|\Delta_{\rm Euc}\chi|\leq Cr^{-2}$. Integration by parts gives
\[
\int_{B(z,r)}|\d u|_{\rm Euc}^2\,\d V_{\rm Euc}
    \leq C\int\chi\Delta_{\rm Euc}\varphi_\tau\,\d V_{\rm Euc}
    =C\int(\varphi_\tau-\varphi_\tau(z))
       \Delta_{\rm Euc}\chi\,\d V_{\rm Euc}
      \leq Cr^{2n-2+2\alpha}.
\]
The constants are independent of $x$ and $\tau$. Morrey's inequality
therefore gives $|u(x')-u(x)|\leq C|x-x'|^\alpha$ on the smaller
ball. Since $u(x)=0$, this proves
\eqref{eq:tangent-distance-modulus}.
\end{proof}

\subsection{Collapse of the fibers and identification of the limit}\label{subsec:fiber-collapse}

\begin{proposition}\label{prop:time-zero-identification}
The map $\Phi_0:(Z_0,d_0^Z)\to Y$ is a homeomorphism. With the metric
\begin{equation}\label{eq:tangent-distance-definition}
    d_Y(y,y')
       =d_0^Z(\Phi_0^{-1}(y),\Phi_0^{-1}(y')),
\end{equation}
the space $(Y,d_Y)$ is proper and geodesic, and the convergence
to $(Y,d_Y,0)$ in \eqref{eq:intro-cone-limit} holds. Moreover,
\begin{equation}\label{eq:tangent-distance-convergence}
    d_t(x,x')\longrightarrow d_Y(\pi(x),\pi(x'))
\end{equation}
uniformly on compact subsets of $X\times X$.
\end{proposition}

\begin{proof}
By \eqref{eq:time-zero-distance} and
\eqref{eq:tangent-distance-modulus}, $\Phi:X\rightarrow Z_0$ is continuous for $d_0^Z$.
By \eqref{eq:time-zero-radius}, the inverse image of each closed ball
centered at $\bar p$ is a compact sublevel set of $F_0$. Surjectivity
of $\Phi$ makes every such ball compact, so $(Z_0,d_0^Z)$ is proper.
By \cite[Lemma~5.6]{LZ26}, its metric and spacetime topologies
coincide. Since the slice is locally compact,
\cite[Theorem~1.6]{LZ26} gives convergence of the full family
$(X,g_t,p)$ to $(Z_0,d_0^Z,\bar p)$ in the pointed
Gromov--Hausdorff sense. In particular, the limit is geodesic.

We next show that $\Phi$ is constant on each fiber of $\pi$.
The morphism $\pi$ is projective with connected fibers, $X$ is smooth,
and $-K_X$ is relatively ample. Thus every fiber is rationally chain
connected by \cite[Corollary~1.4]{HaconMcKernan}.
It suffices to prove that $\Phi\circ u$ is constant for every
holomorphic map $u:\mathbb P^1\to\pi^{-1}(y)$.
Fix $z_0\in\mathbb P^1$ and set
\[
    u_\tau(z)=d_{-\tau}(u(z_0),u(z)).
\]
By \eqref{eq:tangent-distance-modulus} and compactness of the image
of $u$, these functions are uniformly equicontinuous.
Since distance is $1$-Lipschitz, \eqref{eq:tangent-kahler-potential}
and Stokes' theorem give
\begin{equation}\label{eq:curve-vanishing-energy}
    \int_{\mathbb P^1}|\nabla u_\tau|^2\,\d V
       \leq2\int_{\mathbb P^1}u^*\omega_{-\tau}
       =2\tau\int_{\mathbb P^1}u^*\omega\longrightarrow0,
\end{equation}
where the gradient and volume form on $\mathbb P^1$ are taken with
respect to a fixed Fubini--Study metric. Poincar\'e's inequality and
equicontinuity imply that the oscillation of $u_\tau$ tends to zero.
Since $u_\tau(z_0)=0$, \eqref{eq:time-zero-distance} proves the claim.

Thus $\Phi$ descends to a map $\iota:Y\to Z_0$ with
$\Phi=\iota\circ\pi$. The identity $\pi=\Phi_0\circ\Phi$ and the
surjectivity of $\pi$ and $\Phi$ show that $\iota$ and $\Phi_0$ are
inverse maps. The proper surjection $\pi$ is a quotient map, so
$\iota$ is continuous. Continuity of $\Phi_0$ follows from
\cite[Lemma~6.1]{LZ26} and the equality of the metric and spacetime
topologies. This proves the homeomorphism assertion.

The convergence in \eqref{eq:tangent-distance-convergence} is
pointwise by \eqref{eq:time-zero-distance} and uniform on compact
sets by \eqref{eq:tangent-distance-modulus}. Finally, $\psi_t$ fixes
$p$ and is an isometry from $(X,g_t,p)$ to $(X,(-t)g,p)$.
Taking $-t=r^{-2}$ gives this convergence for the chosen minimum
point $p$. For any other fixed base point $p'$, the distance
$r^{-1}d_g(p,p')$ tends to zero, so the pointed limit is unchanged.
\end{proof}

\begin{corollary}\label{cor:uniform-fiber-diameter}
There is a constant $C_n$ depending only on $n$ such that
\[
    \operatorname{diam}_{(X,g)}\pi^{-1}(y)
       \leq C_n\qquad\text{for every }y\in Y.
\]
The diameter is measured using the ambient distance of $(X,g)$.
\end{corollary}

\begin{proof}
Since $g_{-1}=g$, \cite[Lemma~5.4(2)]{LZ26}, in our normalization,
gives at time $-1$
\[
    d_g(x,x')\leq d_0^Z(\Phi(x),\Phi(x'))+C_n.
\]
The constant depends only on $n$ by Theorem~\ref{thm:main} and
\eqref{eq:automatic-entropy}. The first term vanishes when
$\pi(x)=\pi(x')$, by Proposition~\ref{prop:time-zero-identification}.
\end{proof}

\section{Metric cone structure on the tangent space at infinity}\label{sec:metric-cone}

We retain the notation of Section~\ref{sec:time-zero-fibration} and
identify $Y$ with $(Z_0,d_0^Z)$. The function $F_0$ descends to $Y$, and
\eqref{eq:time-zero-radius} becomes
\begin{equation}\label{eq:tangent-radius}
    2F_0(y)=d_Y(0,y)^2.
\end{equation}
For $r>0$, let $D_r$ be the map induced on $Y$ by
$\psi_{-r^{-2}}$. We use the same notation for this map on $X$.
Self-similarity and \eqref{eq:tangent-distance-convergence} give
\begin{equation}\label{eq:tangent-homothety}
    d_Y(D_r y,D_r y')=r d_Y(y,y'),
    \qquad F_0(D_r y)=r^2F_0(y).
\end{equation}
The generator of $s\mapsto D_{e^s}$ on $X$ is
$\nabla^{g_t}F(\cdot,t)=\nabla^g f$. Integrating
$|\d F|_{g_t}^2\le2F$ along its orbits and letting $t\uparrow0$
gives an upper bound for their lengths. The reverse triangle inequality
and \eqref{eq:tangent-radius} give the opposite bound, so
\begin{equation}\label{eq:tangent-radial-distance}
    d_Y(D_r y,D_{r'}y)
       =|r-r'|\sqrt{2F_0(y)}.
\end{equation}

\subsection{The entropy evolution inequality}\label{subsec:entropy-evi}

For $0<\tau\leq1$, consider the probability measure
\[
    \d m_\tau=(2\pi\tau)^{-n}
       e^{-\tau^{-1}F(\cdot,-\tau)-\muent(g)}\,\d V_{g_{-\tau}}.
\]
Write $\mathcal P_2(X,d_{-\tau})$ for the probability measures with
finite second moment, and $W_{2,-\tau}$ for their quadratic
Wasserstein distance; see \cite[Definitions~6.1 and~6.4]{Villani2009}.
Define the scaled relative entropy
\[
    \mathrm{Ent}_\tau(\sigma)
       =\tau\int_X\log\frac{\d\sigma}{\d m_\tau}\,\d\sigma,
       \qquad \sigma\in\mathcal P_2(X,d_{-\tau}),
\]
with value $+\infty$ when $\sigma$ is not absolutely continuous with
respect to $m_\tau$ or has infinite relative entropy.
For fixed $\tau$ and $\varphi\in C_c^\infty(X)$, define
$P_s^\tau\varphi=u(\cdot,s)$, where $u$ is the unique bounded
solution on $X\times[0,\infty)$ of
\[
    \partial_su=\tau\Delta_{g_{-\tau}}u
       -\langle\nabla^{g_{-\tau}}F(\cdot,-\tau),\nabla u\rangle
       =:L_\tau u,\qquad u(\cdot,0)=\varphi.
\]
The operator $L_\tau/\tau$ is the weighted Laplacian with respect to $m_\tau$.
On probability measures, $P_s^\tau$ is defined by
\[
    \int_X\varphi\,\d(P_s^\tau\sigma)
       =\int_XP_s^\tau\varphi\,\d\sigma,
       \qquad\varphi\in C_c^\infty(X),\quad s\geq0.
\]

\begin{lemma}\label{lem:tangent-entropy-evi}
The metric measure space $(X,d_{-\tau},m_\tau)$ is
$\mathrm{RCD}(\tau^{-1},\infty)$. For
$\sigma_0,\sigma\in\mathcal P_2(X,d_{-\tau})$ with finite entropy,
\begin{equation}\label{eq:tangent-entropy-evi}
    \frac12\frac{\d}{\d s}
       W_{2,-\tau}^2(P_s^\tau\sigma_0,\sigma)
       +\frac12W_{2,-\tau}^2(P_s^\tau\sigma_0,\sigma)
    \leq\mathrm{Ent}_\tau(\sigma)
       -\mathrm{Ent}_\tau(P_s^\tau\sigma_0)
\end{equation}
for almost every $s>0$.
Moreover, for all $\sigma,\sigma'\in\mathcal P_2(X,d_{-\tau})$,
\begin{equation}\label{eq:tangent-W2-contraction}
    W_{2,-\tau}(P_s^\tau\sigma,P_s^\tau\sigma')
       \leq e^{-s}W_{2,-\tau}(\sigma,\sigma'),
       \qquad s\geq0.
\end{equation}
\end{lemma}

\begin{proof}
The metric $g_{-\tau}$ is complete. The quadratic growth of the
soliton potential gives finite second moment for $m_\tau$, which
has full support. The soliton equation gives the weighted Ricci
identity
\[
    \Ric(g_{-\tau})+
       \nabla^2\bigl(F(\cdot,-\tau)/\tau\bigr)
       =\tau^{-1}g_{-\tau}.
\]
Hence $(X,d_{-\tau},m_\tau)$ is
$\mathrm{RCD}(\tau^{-1},\infty)$; see
\cite[Theorem~4.9]{Sturm2006} and
\cite[Remark~3.2 and Theorem~5.1(ii)]{AGS}.

By \cite[Theorem~5.1(iii) and~(5.1)]{AGS}, the weighted heat flow
with generator $L_\tau/\tau$ is the $\mathrm{EVI}_{1/\tau}$ gradient
flow of relative entropy with respect to $m_\tau$.
Replacing its time parameter by $\tau s$ gives $P_s^\tau$.
Multiplying the EVI by $\tau$ therefore gives
\eqref{eq:tangent-entropy-evi}.
Finally, the contraction estimate \cite[(6.1)]{AGS}, applied at
time $\tau s$ with curvature bound $\tau^{-1}$, gives
\eqref{eq:tangent-W2-contraction} for all
$\sigma,\sigma'\in\mathcal P_2(X,d_{-\tau})$.
\end{proof}

\begin{lemma}\label{lem:tangent-entropy-recovery}
There is a constant $C_n$, depending only on $n$, such that
for $0<\tau\leq1$ and $\sigma\in\mathcal P_2(X,d_{-\tau})$,
\begin{equation}\label{eq:tangent-entropy-lower}
    \mathrm{Ent}_\tau(\sigma)
       \geq(1-\tau)\int_X F(\cdot,-\tau)\,\d\sigma
          -C_n\tau(1+|\log\tau|).
\end{equation}
For each $x\in X$, fix a coordinate chart centered at $x$ and,
for sufficiently small $\tau>0$, let
\[
    \d\delta_{x;\tau}
       =\frac{\mathbf 1_{B_{\rm Euc}(x,\tau)}}
          {\Vol_{g_{-\tau}}(B_{\rm Euc}(x,\tau))}
          \,\d V_{g_{-\tau}}.
\]
Then
\begin{equation}\label{eq:tangent-entropy-dirac}
    W_{2,-\tau}(\delta_{x;\tau},\delta_x)\longrightarrow0,
    \qquad\mathrm{Ent}_\tau(\delta_{x;\tau})\longrightarrow F_0(x)
    \quad\text{as }\tau\downarrow0.
\end{equation}
\end{lemma}

\begin{proof}
The scalar curvature bound and volume evolution give
\begin{equation}\label{eq:tangent-volume-bounds}
    \tau^{C_n}\omega^n\leq\omega_{-\tau}^n\leq\omega^n,
    \qquad0<\tau\leq1.
\end{equation}
Together with \eqref{eq:tangent-potential-estimates}, this yields
\[
    Z_\tau
       :=\int_X e^{-F(\cdot,-\tau)}\,\d V_{g_{-\tau}}
       \leq(2\pi)^n e^{\muent(g)+C_n}<\infty.
\]
Let $\widehat m_\tau=Z_\tau^{-1}
e^{-F(\cdot,-\tau)}\,\d V_{g_{-\tau}}$. For a measure of finite
entropy, changing the reference probability measure gives
\[
\begin{aligned}
    \mathrm{Ent}_\tau(\sigma)
       &=\tau\int_X\log\frac{\d\sigma}{\d\widehat m_\tau}
            \,\d\sigma
         +(1-\tau)\int_X F(\cdot,-\tau)\,\d\sigma
       +\tau\log\frac{(2\pi\tau)^n e^{\muent(g)}}{Z_\tau}\\
       &\geq(1-\tau)\int_X F(\cdot,-\tau)\,\d\sigma
          -C_n\tau(1+|\log\tau|),
\end{aligned}
\]
since relative entropy between probability measures is nonnegative
and
\(
    \log\frac{(2\pi\tau)^n e^{\muent(g)}}{Z_\tau}
       \geq n\log\tau-C_n.
\)
This proves \eqref{eq:tangent-entropy-lower}.

By \cite[Lemma~5.4(2)]{LZ26} and the continuity of
$\Phi:X\to(Z_0,d_0^Z)$ established in
Proposition~\ref{prop:time-zero-identification},
\[
    W_{2,-\tau}(\delta_{x;\tau},\delta_x)
       \leq\sup_{z\in B_{\rm Euc}(x,\tau)}d_{-\tau}(x,z)
       \leq\sup_{z\in B_{\rm Euc}(x,\tau)}
          d_0^Z(\Phi(x),\Phi(z))+C_n\sqrt\tau\longrightarrow0.
\]
Smoothness of the reference volume form and
\eqref{eq:tangent-volume-bounds} imply
\[
    C_x^{-1}\tau^{2n+C_n}
       \leq\Vol_{g_{-\tau}}(B_{\rm Euc}(x,\tau))
       \leq C_x\tau^{2n}.
\]
The entropy of the normalized volume measure is
\[
    \mathrm{Ent}_\tau(\delta_{x;\tau})
       =\frac{\int_{B_{\rm Euc}(x,\tau)}F(\cdot,-\tau)\,\d V_{g_{-\tau}}}
              {\Vol_{g_{-\tau}}(B_{\rm Euc}(x,\tau))}
          +\tau\log\frac{(2\pi\tau)^n e^{\muent(g)}}
                            {\Vol_{g_{-\tau}}(B_{\rm Euc}(x,\tau))}.
\]
The logarithmic term tends to zero. By
\eqref{eq:tangent-potential-estimates} and the continuity of $F_0$,
\[
    \sup_{z\in B_{\rm Euc}(x,\tau)}|F(z,-\tau)-F_0(x)|
       \leq C_n\tau+\sup_{z\in B_{\rm Euc}(x,\tau)}|F_0(z)-F_0(x)|
       \longrightarrow0.
\]
Thus the average in the first term tends to $F_0(x)$, proving
\eqref{eq:tangent-entropy-dirac}.
\end{proof}

\begin{lemma}\label{lem:tangent-diffusion}
For $x\in X$ and $s>0$, the semigroup satisfies
\begin{equation}\label{eq:tangent-diffusion-heat}
    P_s^\tau\delta_x
       =\nu_{D_{e^{-s}}x,-\tau e^{-2s};-\tau}.
\end{equation}
For the measures $\delta_{x;\tau}$ in
Lemma~\ref{lem:tangent-entropy-recovery}, as $\tau\downarrow0$,
\begin{equation}\label{eq:tangent-diffusion-concentration}
    W_{2,-\tau}(P_s^\tau\delta_{x;\tau},
       \delta_{D_{e^{-s}}x})\longrightarrow0
\end{equation}
and
\begin{equation}\label{eq:tangent-potential-mean}
    \int_X F(\cdot,-\tau)\,\d(P_s^\tau\delta_{x;\tau})
       \longrightarrow e^{-2s}F_0(x),
\end{equation}
uniformly on every bounded interval of $s\geq0$.
\end{lemma}

\begin{proof}
Fix a smooth compactly supported function $\phi$ and let
\[
    t=t(s)=-\tau e^{-2s},\qquad
    u(t(s),z)=(P_s^\tau\phi)(D_{e^s}z).
\]
Since $D_{e^s}$ is generated by
$\nabla^g f=\nabla^{g_{-\tau}}F(\cdot,-\tau)$, the chain rule gives
\begin{align*}
    \frac{\d}{\d s}u(t(s),z)
       &=\bigl(L_\tau P_s^\tau\phi
          +\langle\nabla F,\nabla P_s^\tau\phi\rangle_{g_{-\tau}}\bigr)(D_{e^s}z)\\
       &=\tau(\Delta_{g_{-\tau}}P_s^\tau\phi)(D_{e^s}z)
        =\tau e^{-2s}\Delta_{g_{t(s)}}u(t(s),z).
\end{align*}
Here we used $D_{e^s}^*g_{-\tau}=e^{2s}g_{t(s)}$.
Dividing by $2\tau e^{-2s}$ yields
$\partial_tu=\Delta_{\omega_t}u$,
with initial value $u(-\tau,\cdot)=\phi$. The heat kernel
representation therefore gives
\[
    P_s^\tau\phi(x)
       =u(-\tau e^{-2s},D_{e^{-s}}x)
       =\int_X\phi\,\d\nu_{D_{e^{-s}}x,-\tau e^{-2s};-\tau}.
\]
By duality, this proves \eqref{eq:tangent-diffusion-heat}.

For $r<t<0$, Theorem~\ref{thm:main} gives
$\scal_{\omega_r}\leq C_n/|r|\leq C_n/(t-r)$.
By \cite[Definition~4.1 and Proposition~4.4]{LW26}, applied to
$2g_t$ and using \eqref{eq:automatic-entropy},
\[
    \int_X d_r(q,z)^2\,\d\nu_{q,t;r}(z)\leq C_n(t-r),
    \qquad q\in X.
\]
Applying this estimate with $q=D_{e^{-s}}x$, $r=-\tau$, and
$t=-\tau e^{-2s}$ and using
\eqref{eq:tangent-diffusion-heat}, we obtain
\[
    W_{2,-\tau}^2(P_s^\tau\delta_x,\delta_{D_{e^{-s}}x})
       \leq C_n\tau(1-e^{-2s}).
\]
By the triangle inequality, \eqref{eq:tangent-W2-contraction}, and
the preceding estimate,
\[
\begin{aligned}
    W_{2,-\tau}(P_s^\tau\delta_{x;\tau},\delta_{D_{e^{-s}}x})
       &\leq W_{2,-\tau}(P_s^\tau\delta_{x;\tau},P_s^\tau\delta_x)
          +W_{2,-\tau}(P_s^\tau\delta_x,\delta_{D_{e^{-s}}x})\\
       &\leq e^{-s}W_{2,-\tau}(\delta_{x;\tau},\delta_x)
          +C_n\sqrt{\tau(1-e^{-2s})}.
\end{aligned}
\]
Together with \eqref{eq:tangent-entropy-dirac}, this proves
\eqref{eq:tangent-diffusion-concentration}, including $s=0$.

Finally, $\sqrt{2F(\cdot,-\tau)}$ is $1$-Lipschitz for $d_{-\tau}$.
Minkowski's inequality gives, for $\sigma\in\mathcal P_2(X,d_{-\tau})$,
\begin{equation}\label{eq:integral and W2}
   \left|\left(\int_X2F(\cdot,-\tau)\,\d\sigma\right)^{1/2}
       -\sqrt{2F(z,-\tau)}\right|
       \leq W_{2,-\tau}(\sigma,\delta_z).
\end{equation}
By \eqref{eq:tangent-potential-estimates} and
\eqref{eq:tangent-homothety},
\[
    0\leq F(D_{e^{-s}}x,-\tau)-e^{-2s}F_0(x)\leq C_n\tau.
\] We take $\sigma=P_s^\tau\delta_{x;\tau}$ and $z=D_{e^{-s}}x$.
Then
\eqref{eq:tangent-potential-mean} follows from \eqref{eq:tangent-diffusion-concentration} and \eqref{eq:integral and W2}.
\end{proof}

\begin{proposition}[Radial evolution inequality]\label{prop:tangent-radial-evi}
For every $y,y'\in Y$ and  for almost every $s\in\mathbb R$,
\begin{equation}\label{eq:tangent-radial-evi}
    \frac12\frac{\d}{\d s}d_Y(D_{e^{-s}}y,y')^2
       +\frac12d_Y(D_{e^{-s}}y,y')^2
       \leq F_0(y')-e^{-2s}F_0(y).
\end{equation}
\end{proposition}

\begin{proof}
First fix $0\leq a<a'$. Choose $x,x'\in X$ with $\pi(x)=y$ and $\pi(x')=y'$.
Apply Lemma~\ref{lem:tangent-entropy-evi} with initial measure
$\delta_{x;\tau}$ and comparison measure $\delta_{x';\tau}$.
Multiply \eqref{eq:tangent-entropy-evi} by $e^s$, integrate from
$a$ to $a'$, and use \eqref{eq:tangent-entropy-lower} to obtain
\[
\begin{aligned}
    \frac12\left[e^sW_{2,-\tau}^2
       (P_s^\tau\delta_{x;\tau},\delta_{x';\tau})\right]_a^{a'}
    \leq\int_a^{a'} e^s\bigg(&\mathrm{Ent}_\tau(\delta_{x';\tau})
       -(1-\tau)\int_XF(\cdot,-\tau)\,\d(P_s^\tau\delta_{x;\tau})\\
       &+C_n\tau(1+|\log\tau|)\bigg)\,\d s.
\end{aligned}
\]
By the triangle inequality for $W_{2,-\tau}$,
\[
    |W_{2,-\tau}(P_s^\tau\delta_{x;\tau},\delta_{x';\tau})
       -d_{-\tau}(D_{e^{-s}}x,x')|
       \leq W_{2,-\tau}(P_s^\tau\delta_{x;\tau},\delta_{D_{e^{-s}}x})
          +W_{2,-\tau}(\delta_{x';\tau},\delta_{x'}).
\]
The right-hand side tends to zero uniformly on $[a,a']$ by
Lemmas~\ref{lem:tangent-entropy-recovery} and~\ref{lem:tangent-diffusion}.
The curve $D_{e^{-s}}x$, $s\in[a,a']$, lies in a fixed compact set.
Thus \eqref{eq:tangent-distance-convergence} shows that the squared
Wasserstein distance converges uniformly to
$d_Y(D_{e^{-s}}y,y')^2$.
On the right-hand side of the integrated inequality,
\eqref{eq:tangent-entropy-dirac} gives convergence of the comparison
entropy, and \eqref{eq:tangent-potential-mean} gives uniform
convergence of the potential integral. Passing to the limit yields
\begin{equation}\label{eq--integral inequality}
\frac12\left[e^s d_Y(D_{e^{-s}}y,y')^2\right]_a^{a'}
       \leq\int_a^{a'} e^s\bigl(F_0(y')-e^{-2s}F_0(y)\bigr)\,\d s. 
\end{equation}
For any $b\in\mathbb R$, replace $y$ by $D_{e^{-b}}y$ and
shift the time parameter by $b$. The group property and
$F_0(D_{e^{-b}}y)=e^{-2b}F_0(y)$ shows that \eqref{eq--integral inequality} holds on every bounded real interval.
By the radial distance identity \eqref{eq:tangent-radial-distance},
$s\mapsto d_Y(D_{e^{-s}}y,y')^2$ is locally Lipschitz.
Differentiation therefore proves the almost-everywhere assertion.
\end{proof}

\subsection{Proof of Theorem~\ref{thm:canonical-cone}}
Proposition~\ref{prop:time-zero-identification} gives the proper
geodesic limit and its identification with $Y$. It remains to prove
that the limit is a metric cone.

Fix $y,y'\in Y$. By \eqref{eq:tangent-radial-distance},
$s\mapsto d_Y(D_{e^s}y,y')^2$ is locally Lipschitz on $\mathbb R$.
Apply \eqref{eq:tangent-radial-evi} to $(y,y')$ at time $-s$ and to $(y',y)$
at time $s$. By the homothety property, the two inequalities yield
\[
    \frac{\d}{\d s}\left(e^{-s}d_Y(D_{e^s}y,y')^2\right)
       =2e^sF_0(y)-2e^{-s}F_0(y')
\]
for almost every $s\in\mathbb R$. Integrating from $0$ to $s$
gives, for every $s\in\mathbb R$,
\begin{equation}\label{eq:tangent-cosine-identity}
    d_Y(D_{e^s}y,y')^2
       =2e^{2s}F_0(y)+2F_0(y')
          +e^s\bigl(d_Y(y,y')^2-2F_0(y)-2F_0(y')\bigr).
\end{equation}

Let $\Sigma=\{y\in Y:2F_0(y)=1\}$. This set is compact by
properness. For $y,y'\in\Sigma$, define
\[
    \theta(y,y')=\arccos\left(1-\tfrac12d_Y(y,y')^2\right).
\]
Since $d_Y(y,y')\le2$, this is well defined with values in
$[0,\pi]$. Equations~\eqref{eq:tangent-homothety} and
\eqref{eq:tangent-cosine-identity} imply
\begin{equation}\label{eq:tangent-cone-distance}
    d_Y(D_r y,D_{r'}y')^2
       =r^2+(r')^2-2rr'\cos\theta(y,y').
\end{equation}
Every $z\in Y\setminus\{0\}$ has the unique representation
$z=D_r y$, where $r=\sqrt{2F_0(z)}>0$ and $y=D_{1/r}z\in\Sigma$.
Set $D_0y=0$. The map $(y,r)\mapsto D_r y$ from
$\Sigma\times[0,\infty)$ onto $Y$ identifies only the points at radius zero.
By \eqref{eq:tangent-radius}, \eqref{eq:tangent-cone-distance} holds
also when either radius is zero. Thus \cite[Corollary~2.3]{LaschosMielke}
shows that $\theta$ is a metric on $\Sigma$. Its topology agrees with
that induced by $d_Y$, and the radial map identifies $Y$ with the
metric cone over $(\Sigma,\theta)$.

Since $(Y,d_Y)$ is geodesic,
\cite[Corollary~I.5.11]{BridsonHaefliger} implies that any two points
of $\Sigma$ at $\theta$-distance less than $\pi$ are joined by a
minimizing $\theta$-geodesic. The complement of the proper analytic
subset $\pi^{-1}(0)$ in $X$ is connected, so its image
$Y\setminus\{0\}\cong(0,\infty)\times\Sigma$ is connected.
Thus $\Sigma$ is connected. Let $d_\Sigma$ be the intrinsic length
metric associated with $\theta$. Connectedness and the preceding
local geodesic property show that $d_\Sigma$ is finite. Moreover,
$d_\Sigma=\theta$ whenever $\theta<\pi$, so $d_\Sigma$ induces the
same topology. Hence $(\Sigma,d_\Sigma)$ is a compact length space,
and
\[
    \theta(y,y')=\min\{d_\Sigma(y,y'),\pi\}.
\]
Consequently, \eqref{eq:tangent-cone-distance} reads
\[
    d_Y(D_r y,D_{r'}y')^2
       =r^2+(r')^2-2rr'\cos\bigl(\min\{d_\Sigma(y,y'),\pi\}\bigr),
\]
which identifies $Y$ with $C(\Sigma,d_\Sigma)$.\qed

Transporting the normal affine variety structure of $Y$ by $\Phi_0$
gives $Z_0$ a polarized affine cone structure.
Since $\pi$ is proper with connected fibers and $Y$ is normal,
$\pi_*\mathcal O_X=\mathcal O_Y$.
Using $\pi=\Phi_0\circ\Phi$ and the homeomorphism $\Phi_0$ then gives
$\mathcal O_{Z_0}=\Phi_*\mathcal O_X$. The comparison maps are
equivariant, and the action generated by $J\nabla f$ preserves the
slice distances. Thus $\Phi$ is the polarized Fano fibration. This
proves \cite[Conjecture~6.7]{LZ26} and the metric-cone assertion of
\cite[Conjecture~6.8]{LZ26}.

The entropy argument also applies to real Ricci shrinkers.

\begin{theorem}\label{thm:real-shrinker-cone}
Let $(M^m,g,f)$ be a Ricci shrinker with
bounded scalar curvature, and let $(Z_0,d_0^Z,\bar p)$ be the
time-zero slice of its canonical Ricci flow, with its canonical vertex
$\bar p$. Assume that $(Z_0,d_0^Z)$ is locally compact. Then, for
every $p\in M$,
\[
    (M,r^{-2}g,p)\longrightarrow(Z_0,d_0^Z,\bar p)
       \qquad\text{as }r\to\infty
\]
in the pointed Gromov--Hausdorff sense, and $(Z_0,d_0^Z)$ is a
proper geodesic metric cone over a compact metric space, with a
single point included as a degenerate cone.
\end{theorem}

\begin{proof}
If $M$ is compact, then the blow-downs converge to a point. By \cite[Lemma~5.4]{LZ26},
the worldline map $\Phi:M\to Z_0$ is surjective and
\[
    d_0^Z(\Phi(x),\Phi(x'))
       =\lim_{t\uparrow0}d_t(x,x')=0
       \qquad\text{for all }x,x'\in M.
\]
Thus $Z_0$ is a single point, giving the degenerate cone case.
We may therefore assume that $M$ is noncompact.
By \cite[Lemma~5.6 and Proposition~5.7]{LZ26}, the worldline map
$\Phi:M\to(Z_0,d_0^Z)$ is continuous and surjective, the slice is
proper and geodesic, and the full family of blow-downs converges to
it. 

The scalar curvature bound gives \eqref{eq:tangent-potential-estimates},
and the potential--distance comparison \cite[(5.2)]{LZ26} and
self-similarity give
\eqref{eq:tangent-radius}--\eqref{eq:tangent-radial-distance} on
$Z_0$, with $D_r$ induced by $\psi_{-r^{-2}}$.
Lemmas~\ref{lem:tangent-entropy-evi}--\ref{lem:tangent-diffusion}
apply in the Riemannian setting: their proofs use the soliton identities,
volume evolution, continuity of $\Phi$ and $F_0$, the distance estimate
\cite[Lemma~5.4(2)]{LZ26}, and the heat kernel estimates
\cite[Proposition~4.4]{LW26}.
The constants may now depend on $m$, the scalar curvature bound,
and the entropy of the fixed shrinker.

The proof of \cite[Proposition~5.7, (5.10)]{LZ26} gives uniform
convergence of $d_t(x,x')$ to $d_0^Z(\Phi(x),\Phi(x'))$ on compact
subsets of $M\times M$. Thus the limit argument in
Proposition~\ref{prop:tangent-radial-evi} also applies on $Z_0$.
The derivation of \eqref{eq:tangent-cone-distance} then applies unchanged,
and \cite[Corollary~2.3]{LaschosMielke} identifies $(Z_0,d_0^Z)$ with
the metric cone over its compact unit link.
\end{proof}

The link in Theorem~\ref{thm:real-shrinker-cone} need not be connected.
For example, the round shrinking cylinder $S^3\times\mathbb R$ has
tangent cone $\mathbb R$, whose unit link consists of two points
at angular distance $\pi$.

\section{Metric completion of the regular part}\label{sec:metric-completion}

We prove Theorem~\ref{thm:completion}. Throughout this section,
$(X^n,\omega,f)$ is a K\"ahler--Ricci shrinker with maximal volume
growth. We retain the notation of
Subsection~\ref{subsec:intro-metric-completion} and
Sections~\ref{sec:heat}--\ref{sec:metric-cone}; in particular,
$g_0$ is the smooth limiting metric on $Y^\circ$ and
$X^\circ=\pi^{-1}(Y^\circ)$. Under the identification of
Proposition~\ref{prop:time-zero-identification}, we have $d_Y=d_0^Z$.

\subsection{Bi-Lipschitz comparison of the two distances}\label{subsec:completion-comparison}

\begin{lemma}\label{lem:intrinsic-completion}
The intrinsic distance $d_{g_0}$ on $Y^\circ$ extends to a proper metric
on $Y$, still denoted by $d_{g_0}$, with the analytic topology. Moreover,
$d_Y\leq d_{g_0}$, and $(Y,d_{g_0})$ is the metric completion of
$(Y^\circ,d_{g_0})$.
\end{lemma}

\begin{proof}
By \eqref{eq:tangent-kahler-potential} and
\eqref{eq:tangent-potential-estimates}, $\omega_{-\tau}$ converges
as a positive current to
$\omega_0=\sqrt{-1}\partial\bar\partial F_0$, whose restriction to
$X^\circ$ is a smooth K\"ahler form.
The potential $F_0$ is locally H\"older continuous by
\eqref{eq:tangent-potential-modulus}. Then the following proof proceed similar to Lemma \ref{lem:tangent-distance-modulus} and Proposition \ref{prop:time-zero-identification}.

We first extend the intrinsic distance continuously to $X$.
For $a\in Y^\circ$ and $L>0$, pull
$u=\min\{d_{g_0}(a,\cdot),L\}$ back to $X^\circ$. Since $u$ is $1$-Lipschitz with respect to $g_0$,
\[
    \sqrt{-1}\partial u\wedge\bar\partial u
       \leq C\omega_0=C\sqrt{-1}\partial\bar\partial F_0.
\]
Since $X\setminus X^\circ$ has zero  2-capacity and $\omega_0$ has finite local mass, $u$ extends to
$W^{1,2}_{\mathrm{loc}}(X)$ with the same inequality of currents.
By \eqref{eq:tangent-potential-modulus}, $F_0$ is locally
$C^{2\alpha}$, where $\alpha$ is the exponent in
Lemma~\ref{lem:tangent-distance-modulus}. The argument in
its proof gives
\[
    |u(x)-u(y)|\leq C_K|x-y|^\alpha,\qquad x,y\in K,
\]
on compact subsets $K$ of sufficiently small coordinate balls,
with $C_K$ independent of $a$ and $L$.
 Let $L\to\infty$ and 
this gives the  bound for $d_{g_0}(\pi(x),\pi(y))$x.
The triangle inequality therefore gives a unique finite continuous
extension of $(x,y)\mapsto d_{g_0}(\pi(x),\pi(y))$ from
$X^\circ\times X^\circ$ to a pseudometric on $X$.

We next show that this pseudometric descends to $Y$.
For the functions $u$ above, convolution in a coordinate chart
and Jensen's inequality give
\[
    \sqrt{-1}\partial u_\epsilon\wedge\bar\partial u_\epsilon
       \leq C\sqrt{-1}\partial\bar\partial(F_0)_\epsilon.
\]
Pull back to a holomorphic disk in a fiber. The right-hand side
tends to zero because $F_0$ is constant on the fiber by
\eqref{eq:time-zero-radius}. Thus the
locally uniform limit $u$ is constant on the disk, and continuity
and connectedness give constancy on the whole fiber. Since this
holds for every $a$ and $L$, the extended pseudometric is constant
on fibers in both variables. Properness of $\pi$ makes the descended
pseudometric $d_{g_0}$ continuous on $Y\times Y$.

Finally, smooth convergence on $Y^\circ$ and
\eqref{eq:tangent-distance-convergence} give $d_Y\leq d_{g_0}$ there,
hence on all of $Y$ by continuity. Thus $d_{g_0}$ is a metric.
By Proposition~\ref{prop:time-zero-identification}, $(Y,d_Y)$ is proper
and induces the analytic topology. The inequality and continuity
of $d_{g_0}$ give the same topology for $d_{g_0}$.
Every closed $d_{g_0}$-ball is a closed subset
of a compact $d_Y$-ball, so $d_{g_0}$ is proper and complete.
Density of $Y^\circ$ now gives the claimed metric completion.
\end{proof}

\begin{lemma}\label{lem:trace-estimate}
For every $p\in Y$ and $\tau>0$,
\begin{equation}\label{eq:trace-estimate}
    \int_0^\tau\int_{X^\circ}
       \operatorname{tr}_{\omega_{-s}}\omega_0\,\d\nu_{p;-s}\,\d s\leq n\tau.
\end{equation}
\end{lemma}

\begin{proof}
Choose $x\in\pi^{-1}(p)$ and $0<\delta<\tau$.
By Lemma~\ref{lem:canonical} and Theorem~\ref{thm:main}, on
$\delta\leq s\leq\tau$ the bounded nonnegative function
$F(\cdot,-s)-F_0+n(\tau-s)$ satisfies
\[
    (\partial_s+\Delta_{\omega_{-s}})
       \bigl(F(\cdot,-s)-F_0+n(\tau-s)\bigr)
       =-\operatorname{tr}_{\omega_{-s}}(\sqrt{-1}\partial\bar\partial F_0)
\]
in distributions. Choose nonnegative smooth functions $h_k$, compactly
supported in $X^\circ\times(0,\tau)$, increasing pointwise to
$\operatorname{tr}_{\omega_{-s}}\omega_0$. Fix $k$, take $\delta$
smaller than the time coordinates of its support, and choose a
smooth relatively compact domain $\Omega$ containing $x$ and the
spatial projection of this support. Solve
\[
    (\partial_s+\Delta_{\omega_{-s}})H_k=-h_k
\]
with zero data on the parabolic boundary of
$\Omega\times[\delta,\tau]$, using $\tau-s$ as forward time.
Then $H_k\leq F(\cdot,-s)-F_0+n(\tau-s)$.
Since $F_0$ is locally bounded and plurisubharmonic, it is
locally $W^{1,2}$. The negative part of the
difference is therefore an admissible energy test. Its squared
$L^2$ norm is nonincreasing in forward time; the singular part of
the positive trace current has the required sign.

The scalar curvature bound gives
\[
    F(x,-\delta)-F_0(x)
       =\int_0^\delta\scal_\omega(\psi_{-s}(x))\,\d s
       \leq C_n\delta.
\]
Keeping $h_k$ fixed, Dirichlet exhaustion and evaluation at
$(x,\delta)$ therefore yield
\[
    \int_\delta^\tau\int_X h_k(y,s)\,
       \d\nu_{(x,-\delta);-s}(y)\,\d s
       \leq C_n\delta+n(\tau-\delta).
\]
Let $\delta\to0$. Smooth convergence of the kernels on the fixed
compact support of $h_k$ gives
\[
    \int_0^\tau\int_X h_k(y,s)\,\d\nu_{p;-s}(y)\,\d s
       \leq n\tau.
\]
Letting $k\to\infty$ and applying monotone convergence proves
\eqref{eq:trace-estimate}.
\end{proof}

\begin{lemma}\label{lem:intrinsic-moment}
There is a constant $C$ such that, for every $p\in Y$ and $\tau>0$,
\begin{equation}\label{eq:intrinsic-moment}
    \int_X d_{g_0}(p,\pi(x))^2\,\d\nu_{p;-\tau}(x)\leq C\tau.
\end{equation}
\end{lemma}

\begin{proof}
Fix $p$ and write $\nu_s=\nu_{p;-s}$ and
$\d\nu_s=(2\pi s)^{-n}e^{-b_s}\d V_{g_{-s}}$.
The heat kernel equation and differential Harnack inequality
\cite[Theorem~21]{liwang2020heat1} give
\[
    \partial_s b_s=\Delta b_s-\tfrac12|\d b_s|^2+\scal-n/s,
    \qquad s(\scal+2\Delta b_s-\tfrac12|\d b_s|^2)+b_s-2n\leq0.
\]
Consequently
$\partial_s(\sqrt{s}\,b_s)\leq C_n/(2\sqrt{s})$.
The entropy bound \cite[Theorem~0.1]{ZhangUniform} and the heat kernel
upper bound \cite[Theorem~4.4]{liwang2024heat} give $b_s\geq-A_n$.
Integrating the differential inequality at a fixed point gives that for $t\leq s\leq4t$
$$b_s\leq\sqrt{t/s}\,b_t+C_n(1-\sqrt{t/s})\leq b_t+A_n+C_n.$$
Since $\d V_{g_{-s}}$ increases with $s$,
we obtain for a fixed $c>0$,
\begin{equation}\label{eq:kernel-domination}
    \nu_s\geq c\nu_t,\qquad 0<t\leq s\leq4t.
\end{equation}

Let $u$ be a bounded $1$-Lipschitz function on $(Y,d_{g_0})$.
We use the same symbol for its pullback to $X$ and write
$\nu_s(u)=\int_Xu\,\d\nu_s$ for its mean.
On $X^\circ$, diagonalizing the two Hermitian forms gives
$|\d u|_{g_{-s}}^2\leq\operatorname{tr}_{\omega_{-s}}\omega_0$.
The same weak inequality holds almost
everywhere on $X$. At every finite-energy time, complete-metric
cutoffs and local smoothing put $u$ in the weighted Sobolev domain
of the heat kernel Poincar\'e inequality
\cite[Theorem~3.17]{liwang2024heat}:
\begin{equation}\label{eq:poincare}
    \int|u-\nu_s(u)|^2\,\d\nu_s
       \leq s\int|\d u|_{g_{-s}}^2\,\d\nu_s.
\end{equation}
Let $r_j=2^{-j}r$. By \eqref{eq:trace-estimate}, we can choose
$s_j\in[r_j/2,r_j]$ such that
\[
    \int_{X^\circ}\operatorname{tr}_{\omega_{-s_j}}\omega_0\,\d\nu_{s_j}\leq2n.
\]
Equations~\eqref{eq:kernel-domination}--\eqref{eq:poincare} imply
\[
    |\nu_{s_{j+1}}(u)-\nu_{s_j}(u)|^2
       \leq c^{-1}\int|u-\nu_{s_j}(u)|^2\,\d\nu_{s_j}
       \leq C r_j.
\]
The means therefore converge, and their distance from the limit is
at most $C\sqrt{r_j}$.

 We show that the limit of $\nu_s(u)$ has to be $u(p)$. For a lift $x$ of $p$, the bound from
Theorem~\ref{thm:main},
$\scal(x,a)\leq C_n/|a|\leq C_n/(t-a)$, $a<t<0$, and
\cite[Proposition~4.4]{LW26} give,
\begin{equation}\label{eq:worldline-moment}
    \int_X d_{-s}(x,y)^2\,\d\nu_s(y)\leq Cs.
\end{equation}
By Lemma~\ref{lem:canonical}, $\sqrt{2F(\cdot,-s)}$ is
$1$-Lipschitz. Together with \eqref{eq:tangent-potential-estimates},
this gives
\[
    \int_X F_0\,\d\nu_s\leq2F_0(x)+Cs.
\]
The sets $K_A=\{F_0\leq A\}$ are compact, and
\(
\nu_s(X\setminus K_A)\leq\frac{2F_0(x)+Cs}{A},
\)
so the measures are tight for $0<s\leq1$. For fixed $A$,
\eqref{eq:tangent-distance-convergence} gives
\[
    \sup_{y\in K_A}
       \bigl|d_{-s}(x,y)-d_Y(p,\pi(y))\bigr|\longrightarrow0.
\]
Hence, for every $\epsilon>0$ and sufficiently small $s$,
\eqref{eq:worldline-moment} implies
\[
    (\pi_*\nu_s)\bigl(Y\setminus B_{d_Y}(p,\epsilon)\bigr)
       \leq\frac{2F_0(x)+Cs}{A}+\frac{4Cs}{\epsilon^2}.
\]
First let $s\to0$ and then $A\to\infty$. We obtain
$\pi_*\nu_s\rightharpoonup\delta_p$ in the analytic topology.
By Lemma~\ref{lem:intrinsic-completion}, $u$ is bounded and continuous
in this topology. Therefore
$\nu_s(u)=\int_Yu\,\d(\pi_*\nu_s)\to u(p)$.

We have proved
$\int|u-u(p)|^2\d\nu_{s_j}\leq Cr_j$.
Given $\tau>0$, start with $r=2\tau$. Since
$s_0\in[\tau,2\tau]$, \eqref{eq:kernel-domination} gives
\[
    \int|u-u(p)|^2\,\d\nu_\tau\leq C\tau.
\]
The constant is independent of $p$ and of $\|u\|_\infty$.
Apply this to $u=\min\{d_{g_0}(p,\cdot),L\}$ and let $L\to\infty$.
\end{proof}

\begin{lemma}\label{lem:comparison}
There is a constant $C$ such that, on $Y\times Y$,
\begin{equation}\label{eq:distance-comparison}
    d_Z\leq d_Y\leq d_{g_0}\leq C d_Z.
\end{equation}
\end{lemma}

\begin{proof}
The first inequality follows from the spacetime-distance construction in
Subsection~\ref{subsec:time-zero-fibration} and
\eqref{eq:tangent-distance-convergence}; the second is
Lemma~\ref{lem:intrinsic-completion}. The spacetime-distance comparison
\cite[Proposition~6.20]{FangLiStructure} also gives, for $a'<0$,
\begin{equation}\label{eq:small-spacetime-W1}
    d_{W_1}^{g(a')}(\nu_{p;a'},\nu_{q;a'})\leq d_Z(p,q),
    \qquad p,q\in Y,\quad d_Z(p,q)<\sqrt{-a'/2}.
\end{equation}

Let $r=d_Z(p,q)>0$, $a'=-4r^2$, and $a=-8r^2$.
For $0\leq u\leq1$, the gradient estimate
\cite[Theorem~3.15]{liwang2024heat}, applied to $2g$, gives
\[
    \operatorname{Lip}_{g(a')}(P_{a,a'}u)
       \leq\frac1{\sqrt{2\pi(a'-a)}}.
\]
The estimate extends from smooth to bounded Borel data by heat-kernel
approximation. Kernel reproduction, Kantorovich--Rubinstein duality,
and \eqref{eq:small-spacetime-W1} therefore imply
\[
\left|\int_X u\,\d(\nu_{p;a}-\nu_{q;a})\right|
       =\left|\int_XP_{a,a'}u\,\d(\nu_{p;a'}-\nu_{q;a'})\right|
       \leq\frac{d_{W_1}^{g(a')}(\nu_{p;a'},\nu_{q;a'})}
                     {\sqrt{2\pi(a'-a)}}
        \leq\frac1{\sqrt{8\pi}}.
\]
Thus the measures have a common part $\eta\leq\nu_{p;a},\nu_{q;a}$
with $\eta(X)\geq1/2$.
Integrating the squared triangle inequality for $d_{g_0}$ against
$\eta$ and using \eqref{eq:intrinsic-moment} gives
\[
    \tfrac12d_{g_0}(p,q)^2
       \leq2\int_X d_{g_0}(p,\pi(x))^2\,\d\nu_{p;a}(x)
          +2\int_X d_{g_0}(q,\pi(x))^2\,\d\nu_{q;a}(x)
       \leq Cr^2.
\]
This proves the last inequality. 
\end{proof}

\subsection{Tangent lines and splitting}\label{subsec:completion-model}

Let $(Z,d_Z,\mathfrak t)$ be the spacetime completion from
Subsection~\ref{subsec:time-zero-fibration}.
Its time-zero slice is identified with $Y$, with canonical slice
distance $d_Y$ and conjugate heat kernel measures $\nu_{p;t}$
based at $p\in Y$.
Fix $p\in Y$ and a sequence of  scales
$r_i\downarrow0$. A \emph{tangent flow at $p$ along $(r_i)$}
is a pointed parabolic metric space obtained as a subsequential
pointed Gromov--Hausdorff limit
\begin{equation}\label{eq:model-spacetime-tangent}
    (Z,r_i^{-1}d_Z,p,r_i^{-2}\mathfrak t)
       \xrightarrow{\mathrm{pGH}}(Z',d_{Z'},z',\mathfrak t'),
       \qquad \mathfrak t'(z')=0,
\end{equation}
including convergence of the time functions. This is
\cite[Definition~7.8]{FangLiStructure}, applied to the completed flow. On negative
times the rescaled metrics and distinguished conjugate heat flow are
\begin{equation}\label{eq:model-unshifted-rescaling}
    g_{i,t}=r_i^{-2}g_{r_i^2t},\qquad
    \omega_{i,t}=r_i^{-2}\omega_{r_i^2t},\qquad
    \nu^i_{p;t}=\nu_{p;r_i^2t},\qquad t<0.
\end{equation}

Theorem~\ref{thm:main} allows us to apply the spacetime
framework to the complete shrinker flow by \cite[Section~5]{LZ26}.
The regular part of the limit is a Ricci flow spacetime, and
\cite[Theorem~5.2]{FangLiStructure} gives smooth convergence of the
metrics and time vector fields through time-preserving embeddings
on compact regular regions. The associated backward metric flow
embeds isometrically into $Z'$ by
\cite[Theorem~4.27]{FangLiStructure}. Its conjugate heat kernel
measures arise as limits of the rescaled measures, while their
densities converge smoothly on the earlier regular spacetime by
\cite[Theorem~5.20]{FangLiStructure}. The parallel complex structures
also converge smoothly after passing to a subsequence.

Write $\NN_p(0)=\lim_{\tau\downarrow0}\NN_p(\tau)\in[-A_n,0]$.
Entropy convergence \cite[Lemma~7.2]{FangLiStructure} and scaling give,
for every $\tau>0$,
\begin{equation}\label{eq:model-constant-entropy}
    \NN_{z'}(\tau)
       =\lim_i\NN_p(r_i^2\tau)=\NN_p(0).
\end{equation}
Thus $Z'$ is a Ricci shrinker space in the sense of
\cite[Definition~7.16]{FangLiStructure}, and
\cite[Proposition~7.9]{FangLiStructure} gives the shrinking soliton
equation on its regular negative-time part, with connected regular
slices $\RR'_t$. Smooth convergence of the parallel complex
structures makes this regular part K\"ahler. Denote its metric by
$g_{\infty,t}$ and its K\"ahler form by $\omega_{\infty,t}$. If
\[
    \d\nu_{z';t}
       =(2\pi|t|)^{-n}e^{-b_t}\,\d V_{g_{\infty,t}},
\]
then
\[
    \Ric(\omega_{t}^{Z'})+
       \sqrt{-1}\partial\bar\partial b_t
       =|t|^{-1}\omega_{\infty,t}.
\]
 Following \cite[Definition~6.1]{FangLiStructure}, we write
$d_t^{Z'}$ for the canonical slice distance on $Z'_t$.
Write $\RR'$ for the regular spacetime. By
\cite[Corollary~7.10]{FangLiStructure}, the restriction of $d_t^{Z'}$
to $\RR'_t$ agrees with its intrinsic Riemannian distance for every $t<0$.
The earlier kernel measures give full mass to $\RR'_t$ and have
finite second moments. In particular, regular $H$-centers exist,
with $H=H(n)$ fixed; see
\cite[Proposition~5.24 and Definition~5.25]{FangLiStructure}.

For points $p,q\in Y=Z_0$, put
\[
    W_a(p,q)=d_{W_1}^{g(a)}(\nu_{p;a},\nu_{q;a}),\qquad a<0.
\]
 The definition of the spacetime distance and
\eqref{eq:small-spacetime-W1} give
\begin{equation}\label{curve:comparison}
\begin{split}
    d_Z(p,q)&\leq\max\left\{\sqrt{-a/2},W_a(p,q)\right\},\\
   W_a(p,q)&\leq d_Z(p,q)
       \quad\text{if }d_Z(p,q)<\sqrt{-a/2}.
\end{split}
\end{equation}
These inequalities also hold in a tangent flow, with $d_{Z'}$,
$d_t^{Z'}$, and its own kernels in place of the original distances
and kernels; see \cite[Proposition~6.20]{FangLiStructure}.

For a Lipschitz curve $c:[a,a']\to(Y,d)$, its metric speed is
\[
    |c'|_d(t):=\lim_{h\to0}\frac{d(c(t+h),c(t))}{|h|},
\]
whenever the limit exists. By \cite[Theorem~1.1.2]{AGSGradientFlows},
this limit exists for almost every $t\in(a,a')$. The length formula
\cite[Theorem~2.7.6]{BBI} gives
\[
    \operatorname{Length}_d(c)=\int_a^{a'}|c'|_d(t)\,\d t.
\]

\begin{lemma}\label{curve:lengths}
For every $d_Z$-rectifiable curve $c$ in $Y$,
\[
    \operatorname{Length}_{d_Z}(c)=\operatorname{Length}_{d_Y}(c).
\]
If $\gamma$ is a unit-speed $d_Y$-geodesic, then for almost every
interior parameter $t$,
\begin{equation}\label{curve:differentiation}
    d_Z(\gamma(t+h),\gamma(t+k))
       =|h-k|+o(|h|+|h-k|).
\end{equation}
At each such parameter, every tangent flow at $\gamma(t)$ contains
a $d_{Z'}$-isometric line through its basepoint in the time-zero slice,
obtained as a limit of the rescaled curve.
\end{lemma}

\begin{proof}
Fix $a<0$. Subdivide any subarc $c|_{[u,v]}$ so that successive
$d_Z$-distances are less than $\sqrt{-a/2}$. By
\eqref{curve:comparison} and the triangle inequality,
\[
   W_a(c(u),c(v))
       \leq\sum_j d_Z(c(t_j),c(t_{j+1}))
       \leq\operatorname{Length}_{d_Z}(c|_{[u,v]}).
\]
Letting $a\uparrow0$ and then summing over a partition of $c$ gives
$\operatorname{Length}_{d_Y}(c)\leq\operatorname{Length}_{d_Z}(c)$.
The reverse inequality follows from $d_Z\leq d_Y$.

Applied to every subarc of $\gamma$, this equality and the length
formula above give $|\gamma'|_{d_Z}(t)=1$ almost
everywhere. Kirchheim's metric
differentiation theorem \cite[Theorem~2]{Kirchheim}, in the form
\cite[Theorem~3.6]{CR}, gives \eqref{curve:differentiation}.
For fixed $s,s'$, substitute $h=r_i s$ and $k=r_i s'$ into that
identity. The rescaled distances tend to $|s-s'|$. Properness of
the tangent flow and a diagonal subsequence therefore give the
asserted line; convergence of the time functions places it at
time zero.
\end{proof}

\begin{lemma}\label{curve:orbit}
Suppose $\ell:\mathbb R\to Z'_0$ is $d_{Z'}$-isometric and
$\ell(0)=z'$. The regular negative-time flow splits off a flat complex
line. Its translations $T_\zeta$, $\zeta\in\mathbb C$, are holomorphic
on the regular part, extend to $Z'_{(-\infty,0]}$, and satisfy
\begin{equation}\label{curve:orbit-identity}
    \ell(s)=T_s z',\qquad 0\leq s\leq1.
\end{equation}
They preserve $d_{Z'}$, the canonical slice distances, and the
conjugate kernels. On each regular negative-time slice, the real
translation field has unit norm.
\end{lemma}

\begin{proof}
Fix $p,q\in Z'_0$ and $t<0$. Subdivision and \eqref{curve:comparison} give
$$d_{W_1}^{Z'_t}(\ell(s),\ell(s'))\leq|s-s'|,$$ with equality whenever
$|s-s'|>\sqrt{-t/2}$. For a shorter pair $s<s'$, choose $s''>s'$
with $s''-s'>\sqrt{-t/2}$. The triangle inequality gives
\[
    d_{W_1}^{Z'_t}(\ell(s),\ell(s'))
       \geq d_{W_1}^{Z'_t}(\ell(s),\ell(s''))-d_{W_1}^{Z'_t}(\ell(s'),\ell(s''))
       =s'-s.
\]
Thus, for every $t<0$,
\begin{equation}\label{curve:saturation}
    d_{W_1}^{Z'_t}(\ell(s),\ell(s'))=d_0^{Z'}(\ell(s),\ell(s'))=|s-s'|.
\end{equation}

We perform transport on the complete separable slices of the
associated metric soliton. Their regular parts identify with
$\RR'_t$, and their distances agree with $d_t^{Z'}$ on $\RR'_t$;
see \cite[Theorem~4.17(4), Proposition~7.9, and Corollary~7.10]{FangLiStructure}.
Since the kernel measures give full mass to $\RR'_t$, their
Wasserstein distances are unchanged.

Choose an optimal $1$-Lipschitz Kantorovich potential $v$ at
time $-2$, oriented so that
$\int v\,\d(\nu_{\ell(1);-2}-\nu_{z';-2})=1.$
The heat evolution $u(t)=P_{-2,t}v$ is $1$-Lipschitz and smooth
on the regular spacetime for $-2<t<0$. Indeed, bounded
truncations converge locally uniformly under the heat flow by
the local second-moment bounds, and interior estimates give
smooth convergence. Kernel reproduction yields
\begin{equation}\label{curve:dual-equality}
    \int u(-1)\,\d(\nu_{\ell(1);-1}-\nu_{z';-1})=1.
\end{equation}
If $|\nabla u(-1)|<1$ everywhere on $\RR'_{-1}$, then
\[
    |u(-1,y)-u(-1,x)|<d_{-1}^{Z'}(x,y)
\]
for every pair of distinct regular points by \cite[Corollary~7.10]{FangLiStructure}. An optimal coupling of the two kernel measures
has full mass on regular pairs and positive mass off the diagonal,
since its transport cost is one. Integrating the strict inequality
contradicts \eqref{curve:dual-equality}. Hence $|\nabla u|=1$ at a
regular point at time $-1$. Bochner's identity
\[
    (\partial_t-\Delta)|\nabla u|^2=-|\nabla^2u|^2
\]
and the strong maximum principle on regular cylinders, together
with connectedness of the regular slices, show that
$V=\nabla u(-1)$ is a global unit parallel gradient field.

By \cite[Remark~8.3 and Proposition~8.2]{FangLiStructure}, $V$
defines a real splitting coordinate $a$ on the regular negative-time flow,
with $a-u(-1)$ constant on $\RR'_{-1}$ and
\begin{equation}\label{curve:parallel-coordinates}
    \partial_{\mathfrak t'}a=0,\qquad \nabla^2a=0,\qquad |\nabla a|=1.
\end{equation}
The K\"ahler argument in \cite[proof of Lemma~11.7]{FangLiStructure}
shows that $JV$ gives an orthogonal splitting direction. These
directions form a flat complex line, and its translations are
holomorphic because $J$ is parallel. By
\cite[Corollary~5.7, Proposition~8.4, and (8.7)]{FangLiStructure}, the translations
extend to $Z'_{(-\infty,0]}$ and preserve its distances and kernels.
The product heat flow preserves $a$. Therefore kernel reproduction
makes $\int a\,\d\nu_{w;t}$ independent of $t<0$ for each $w\in Z'_0$.
By \eqref{curve:dual-equality} and the fact that \(a-u(-1)\) is constant, these averages at \(\ell(1)\) and \(z'\) differ by one. So the
$1$-Lipschitz bound and \eqref{curve:saturation} give
\begin{equation}\label{curve:coordinate-mean}
    \int a\,\d(\nu_{\ell(s);t}-\nu_{z';t})=s,
       \qquad 0\leq s\leq1,\quad t<0.
\end{equation}

To identify the original line, fix $s\in[0,1]$ and choose regular
$H$-centers $x_t,y_t$ of $z',\ell(s)$ at time $t<0$.
By the definition of an $H$-center, \eqref{curve:saturation}, \eqref{curve:saturation}, and
\eqref{curve:coordinate-mean} imply
\[
    d_t^{Z'}(x_t,y_t)=s+O(\sqrt{-t}),\qquad
    a(y_t)-a(x_t)=s+O(\sqrt{-t}).
\]
The product metric on the regular negative-time slice gives
\[
    d_t^{Z'}(T_sx_t,y_t)^2
       =d_t^{Z'}(x_t,y_t)^2
        -2s\bigl(a(y_t)-a(x_t)\bigr)+s^2\longrightarrow0.
\]
Since $T_sx_t$ is an $H$-center of $T_sz'$, it follows that
\begin{equation}\label{eq-d_w1 distance}
    d_{W_1}^{Z'_t}(T_sz',\ell(s))
       \leq d_t^{Z'}(T_sx_t,y_t)+2\sqrt{H(-t)}\longrightarrow0.
\end{equation}Combining \eqref{eq-d_w1 distance} with \eqref{curve:comparison} and letting $t\uparrow 0$, we obtain \eqref{curve:orbit-identity}.
\end{proof}

\Needspace{6\baselineskip}
\begin{remark}\label{rem:slice-line-splitting}
The argument above also yields a splitting in the purely Riemannian
setting. Let
$(Z,d_Z,\mathfrak t)$ be a noncollapsed Ricci flow limit space,
and suppose that $\ell:\mathbb R\to Z_{t_0}$ is a $d_Z$-isometric
embedding. Let $\mathcal X^{\ell(0)}$ be a metric flow associated
with $\ell(0)$ as in \cite[Section~4, (4.5)]{FangLiStructure}. Then $\mathcal X^{\ell(0)}$ splits off a stationary
$\mathbb R$-factor for $t<t_0$ in its time interval of definition.
The translations extend to the closure in $Z$ of its regular
spacetime up to time $t_0$.

As in Lemma~\ref{curve:orbit}, subdivision along $\ell$ gives
\[
    d_{W_1}^{Z_t}(\nu_{\ell(s);t},\nu_{\ell(s');t})=|s-s'|,
    \qquad t<t_0.
\]
The kernels along the line belong to the same associated backward
regular spacetime by \cite[Lemma~6.10]{FangLiStructure}.
Choose the initial and intermediate times outside the exceptional
set in \cite[Proposition~6.23]{FangLiStructure}. At these times,
the canonical slice distance on this regular component agrees
with its intrinsic distance and with the associated metric-flow
distance; see also \cite[Theorem~4.17(4)]{FangLiStructure}.
The same transport and Bochner argument gives time-independent
functions with unit parallel gradients on exhausting backward
time intervals. Normalize them at a fixed regular point and
pass to a locally smooth subsequential limit to obtain a function
$u$ on the regular part of $\mathcal X^{\ell(0)}$ for $t<t_0$ with
$|\nabla u|=1$, $\nabla^2u=0$, and $\partial_{\mathfrak t}u=0$.
Theorem~4.19 of \cite{FangLiStructure} gives the global metric-flow
splitting. Its translations preserve the slice distances and kernels,
hence the spacetime distance under the embedding of
\cite[Theorem~4.27]{FangLiStructure}, and extend by continuity
to the closure in $Z$.
\end{remark}

\subsection{Recovery of regular curves}\label{subsec:completion-curves}

We use the complex Euclidean factor at a negative time to construct curves
in $Y^\circ$. 
The canonical soliton identities \eqref{eq:F-identities} and
\eqref{eq:Kahler-Hess} give
\[
    \sqrt{-1}\partial\bar\partial F(\cdot,-\tau)
       =\omega_{-\tau}-\tau\Ric(\omega_{-\tau}).
\]
Subtracting this identity from
$\omega_0=\sqrt{-1}\partial\bar\partial F_0$ and using
\eqref{eq:tangent-potential-estimates}, we obtain
\begin{equation}\label{eq:completion-current}
\begin{split}
    \omega_0&=\omega_{-\tau}-\tau\Ric(\omega_{-\tau})
       +\sqrt{-1}\partial\bar\partial
          \bigl(F_0-F(\cdot,-\tau)\bigr),\\
    0&\leq F(\cdot,-\tau)-F_0\leq C_n\tau.
\end{split}
\end{equation}
For the endpoint estimate, use time $-\tau$ as the comparison
time in the definition of $d_Z$. The bound
\eqref{eq:worldline-moment} then gives
\begin{equation}\label{eq:completion-worldlines}
d_Z((x,-\tau),\pi(x))
       \leq\max\left\{\sqrt{\tau/2},
          d_{W_1}^{g(-\tau)}(\delta_x,\nu_{\pi(x);-\tau})\right\}
       \leq C\sqrt{\tau},\qquad x\in X,\quad\tau>0.
\end{equation}

\begin{proposition}\label{prop:regular-curve-recovery}
Let $\gamma:[0,L]\to(Y,d_Y)$ be a unit-speed geodesic.
Fix $a\in(0,L)$ satisfying \eqref{curve:differentiation}, and
choose scales $r_i\downarrow0$ giving a tangent flow at $\gamma(a)$.
For every $\epsilon>0$, after passing to a subsequence there are
smooth curves $c_i:[0,1]\to Y^\circ$ such that
\begin{equation}\label{curve:recovery}
\begin{split}
    \operatorname{Length}_{g_0}(c_i)&\leq r_i(1+\epsilon+o_i(1)),\\
    d_Z(c_i(0),\gamma(a))&\leq\epsilon r_i+o_i(r_i),\\
    d_Z(c_i(1),\gamma(a+r_i))&\leq\epsilon r_i+o_i(r_i).
\end{split}
\end{equation}
\end{proposition}

\begin{proof}
By Lemmas~\ref{curve:lengths} and~\ref{curve:orbit}, we may
pass to a subsequence so that the rescaled geodesic converges
to a line $\ell$ with $\ell(x)=T_xz'$ for $0\leq x\leq1$.
For $\delta>0$, choose a regular $H$-center
$q_\delta\in\RR'_{-\delta}$ of $z'$. By
\cite[Lemma~5.27]{FangLiStructure},
\[
    d_{Z'}(q_\delta,z')\leq C\sqrt{\delta}.
\]
Normalize the coordinate $z=x+\sqrt{-1}y$ on the flat complex
factor so that its metric and K\"ahler form are
$\d x^2+\d y^2$ and $\d x\wedge\d y$.

For $h,\sigma>0$, choose a rectangle $Q\subset\mathbb C$
containing a neighborhood of $[0,1]\times[-h,h]$ and a cutoff
$\zeta\in C_c^\infty(Q)$ with $0\leq\zeta\leq1$ such that
\begin{equation}\label{eq:completion-strip-cutoff}
    \zeta=1\ \text{on }[0,1]\times[-h,h],\qquad
    \int_Q\zeta\,\d x\,\d y\leq2h(1+\sigma).
\end{equation}
Write
$C_\zeta
=\int_Q|\sqrt{-1}\partial\bar\partial\zeta|$.
These choices depend only on $h$ and $\sigma$.
Choose a transverse holomorphic coordinate ball
$B\subset\mathbb C^{n-1}$, with parametrization $q_\delta(w)$,
such that $q_\delta(0)=q_\delta$ and
$d_{-\delta}^{Z'}(q_\delta(w),q_\delta)<h$.
The product decomposition gives a holomorphic chart
\[
    \Psi_\infty:Q\times B\longrightarrow\RR'_{-\delta},
    \qquad \Psi_\infty(z,w)=T_zq_\delta(w),
\]
defined on a neighborhood of the closure. On the strip,
\begin{equation}\label{eq:completion-chart-close}
    d_{Z'}(\Psi_\infty(x+\sqrt{-1}y,w),\ell(x))
       \leq C\sqrt{\delta}+2h.
\end{equation}
The extent of the chart in the complex factor is independent of
$\delta$ and of the transverse radius.

For each fixed $\delta$ and chart, smooth regular convergence
\cite[Theorem~5.2]{FangLiStructure}, in the setting of
\cite[Section~5]{LZ26}, gives holomorphic embeddings
$\Psi_i:Q\times B\to X$ such that
\[
    \Psi_i^*\omega_{i,-\delta}
       \longrightarrow\Psi_\infty^*\omega_{\infty,-\delta}
\]
smoothly on a neighborhood of the compact strip.
To obtain holomorphic convergence maps, pull back the complex
structures by the smooth convergence maps. The resulting $J_i$
converge smoothly to the standard structure. On nested strictly
convex domains inside a slightly larger product chart, solve
$\bar\partial_{J_i}u_i^j=\bar\partial_{J_i}z^j$ with
$\|u_i^j\|_{L^2}\to0$, using the local $L^2$ estimate
\cite{Hormander}. A fixed strictly plurisubharmonic weight gives
uniform constants, and interior elliptic estimates give smooth
convergence to zero on a smaller domain. The corrected coordinates
$(z^j-u_i^j)_j$ converge smoothly to the identity. Their inverses,
composed with the original convergence maps, give $\Psi_i$.

Since the complex structure on $X$ is fixed along the canonical
flow, we can pull back \eqref{eq:completion-current} with
$\tau=r_i^2\delta$ by the same maps. Set
\begin{equation}\label{eq:completion-rescaled-current}
\begin{split}
    \Theta_i&:=\Psi_i^*(r_i^{-2}\omega_0)
=A_i+\sqrt{-1}\partial\bar\partial v_i,\\
    A_i&=\Psi_i^*\bigl(\omega_{i,-\delta}
                -\delta\Ric(\omega_{i,-\delta})\bigr),\\
    v_i&=r_i^{-2}\bigl(F_0-F(\cdot,-r_i^2\delta)\bigr)\circ\Psi_i,
       \qquad \|v_i\|_\infty\leq C_n\delta.
\end{split}
\end{equation}
By the smooth convergence, we know that the restrictions of $A_i$ to the $z$-disks
converge uniformly to $\d x\wedge\d y$, also uniformly in a
smaller transverse ball.

Continuous potentials allow us to restrict the positive current
$\Theta_i$ to each holomorphic $z$-disk. By distributional
integration by parts and \eqref{eq:completion-strip-cutoff},
\begin{equation}\label{eq:completion-strip-area}
\begin{split}
    \Theta_i|_w([0,1]\times[-h,h])
       &\leq\int_Q\zeta\,\Theta_i|_w=\int_Q\zeta\,A_i|_w
          +\int_Qv_i(z,w)\sqrt{-1}\partial\bar\partial\zeta\\
       &\leq2h(1+\sigma)+C_nC_\zeta\delta+o_i(1).
\end{split}
\end{equation}

The curve-selection step below is a local variant of the averaging
argument of Demailly--Peternell--Schneider
\cite[proof of Lemma~1.3]{DPS93}. We apply it on holomorphic strips,
using \eqref{eq:completion-strip-area} to control their areas.
Let $E=X\setminus X^\circ$. Since $\Psi_i$ is a holomorphic embedding,
$\Psi_i^{-1}(E)$ is a proper analytic subset. For almost every
$w$, its intersection with the $z$-disk is discrete and finite
on the compact strip. All but finitely many horizontal segments
therefore avoid it. On the regular portion of such a disk, write
the metric of $\Theta_i$ as $a_i(x,y)(\d x^2+\d y^2)$.
By Cauchy--Schwarz and \eqref{eq:completion-strip-area}, we have
\[
    \frac1{2h}\int_{-h}^h\int_0^1
       \sqrt{a_i(x,y)}\,\d x\,\d y
       \leq\left(1+\sigma+\frac{C_nC_\zeta\delta}{2h}
          +o_i(1)\right)^{1/2}.
\]
Therefore there is a segment
$x\mapsto(x+\sqrt{-1}y_i,w_i)$ avoiding $\Psi_i^{-1}(E)$.
Its image under $\pi\circ\Psi_i$ is a smooth curve $c_i$ in
$Y^\circ$ satisfying
\begin{equation}\label{eq:completion-strip-length}
    r_i^{-2}\operatorname{Length}_{g_0}(c_i)^2
       \leq1+\sigma+\frac{C_nC_\zeta\delta}{2h}+o_i(1).
\end{equation}

For the endpoints, \eqref{eq:completion-worldlines} gives
\[
    r_i^{-1}d_Z((\Psi_i(z,w),-r_i^2\delta),
                      \pi(\Psi_i(z,w)))\leq C\sqrt{\delta}.
\]
Smooth regular convergence is compatible with the spacetime
convergence, uniformly on compact subsets of the chart, by
\cite[Theorem~5.2(c)]{FangLiStructure}. Combining this with
\eqref{eq:completion-chart-close} gives
\begin{equation}\label{eq:completion-strip-endpoints}
\begin{split}
    r_i^{-1}d_Z(c_i(0),\gamma(a))
       &\leq C'\sqrt{\delta}+2h+o_i(1),\\
    r_i^{-1}d_Z(c_i(1),\gamma(a+r_i))
       &\leq C'\sqrt{\delta}+2h+o_i(1).
\end{split}
\end{equation}
First choose $h<\epsilon/4$ and $\sigma>0$ so that
$\sqrt{1+2\sigma}<1+\epsilon$, and fix $Q$, $\zeta$, and $C_\zeta$.
Next choose $\delta>0$ so that
$C'\sqrt{\delta}<\epsilon/4$ and
$C_nC_\zeta\delta/(2h)<\sigma$.
With these choices fixed, let $i\to\infty$.
Equations~\eqref{eq:completion-strip-length} and
\eqref{eq:completion-strip-endpoints} prove \eqref{curve:recovery}.
\end{proof}

\begin{proof}[Proof of Theorem~\ref{thm:completion}]
Let $\gamma:[0,L]\to Y$ be a unit-speed $d_Y$-geodesic.
Lemma~\ref{lem:comparison} makes $\gamma$ $C$-Lipschitz for
$d_{g_0}$. Fix an interior parameter $a$ where
$|\gamma'|_{d_{g_0}}(a)$ exists and
\eqref{curve:differentiation} holds; this is true almost everywhere
by Lemma~\ref{curve:lengths}.
Choose scales $r_i\downarrow0$ giving a tangent flow at $\gamma(a)$.
For $\epsilon>0$, Proposition~\ref{prop:regular-curve-recovery}
gives curves satisfying \eqref{curve:recovery}.
Using Lemma~\ref{lem:comparison} for the two endpoint errors,
we obtain
\[
    d_{g_0}(\gamma(a),\gamma(a+r_i))
       \leq\bigl(1+(1+2C)\epsilon+o_i(1)\bigr)r_i.
\]
Letting $i\to\infty$ and then $\epsilon\downarrow0$ gives
$|\gamma'|_{d_{g_0}}(a)\leq1$.
The length formula therefore implies
\[
    d_{g_0}(\gamma(0),\gamma(L))
       \leq\int_0^L|\gamma'|_{d_{g_0}}(a)\,\d a
       \leq L=d_Y(\gamma(0),\gamma(L)).
\]
Together with $d_Y\leq d_{g_0}$ and
Lemma~\ref{lem:intrinsic-completion}, this proves the theorem.
\end{proof}

\appendix
\titleformat{\section}{\normalfont\Large\bfseries}{Appendix~\thesection.}{0.75em}{}
\section{Metric completion for noncollapsed K\"ahler--Ricci flows}\label{app:finite-time}

The argument of Section~\ref{sec:metric-completion}  also applies to finite-time singularities
of compact volume-noncollapsed K\"ahler--Ricci flows.

Let $(X^n,\omega_t)_{-1\leq t<0}$ be a K\"ahler--Ricci flow on a compact
connected K\"ahler manifold, with first singular time zero and
$\lim_{t\uparrow0}\Vol(X,\omega_t)>0$. The limiting class $\alpha=[\omega_{-1}]+2\pi c_1(K_X)$ is nef and big. By
\cite{HaconXie}, there exists a bimeromorphic projective contraction with connected
fibers $\pi:X\to Y$ onto a normal compact K\"ahler space and a K\"ahler
form $\eta$ with $\alpha=\pi^*[\eta]$.  
Moreover, $E:=\operatorname{Null}(\alpha)=\operatorname{Exc}(\pi)$.
Set $X^\circ=X\setminus E$ and $Y^\circ=Y\setminus\pi(E)$.
By \cite{CollinsTosatti}, the metrics converge smoothly
to a K\"ahler metric $g_0$ on compact subsets of $X^\circ\simeq Y^\circ$.

The following theorem establishes the metric-completion assertion in
the Song--Tian program \cite[Section~6.2]{SongTian17} at the first
volume-noncollapsed singular time of a compact K\"ahler--Ricci flow.
\begin{theorem}[Metric completion at finite-time singularities]\label{thm:finite-time-completion}
 The singular-time Gromov--Hausdorff limit is 
homeomorphic to $Y$, and is 
isometric to the metric completion of $(Y^\circ,d_{g_0})$.
\end{theorem}

For compact volume-noncollapsed
K\"ahler--Ricci flows, \cite[Theorem~1.1]{LZ26} identifies the
singular-time Gromov--Hausdorff limit with the canonical time-zero
slice of the spacetime completion. This slice is also isometric to Bamler's
asymptotic boundary \cite[Definition~2.32]{bamler3}; see
\cite[Lemma~9.1]{FangLiStructure}.
More generally,
\cite[Remark~5.8]{LZ26} gives the same convergence for a compact
Ricci flow with a Type~I scalar curvature bound, provided that its
time-zero slice is locally compact in the slice metric. For general compact K\"ahler--Ricci flows, Jian--Song
\cite[Theorem~1.2]{JS26b} proved that the singular-time
Gromov--Hausdorff limit is homeomorphic to the contraction space
determined by the limiting class; see also
\cite{ZZ26} for related results.  The additional
conclusion here is the isometry with the intrinsic metric completion of its regular part.

Put $\theta=\pi^*\eta$ and $\chi=\omega_{-1}-\theta$, and choose a
smooth volume form $\Omega$ with $\Ric(\Omega)=\chi$. We have
\[
 \omega_t=\theta-t\chi+\sqrt{-1}\partial\bar\partial\varphi_t,
 \qquad \dot\varphi_t=\log\frac{\omega_t^n}{\Omega},
 \qquad \varphi_{-1}=0.
\]
The scalar maximum principle and \cite[Theorem~1.1]{JS26} give
$-C_0\leq\scal_{\omega_t}\leq K/|t|$.
The following identities are well known; see, for example,
\cite[Lemma~4.1 and Proposition~2.1]{JS26}.

\begin{lemma}\label{lem:finite-time-potential}
For $H_t=\varphi_t-t\dot\varphi_t$,
\begin{equation}\label{eq:finite-time-H}
\partial_tH_t=t\scal_{\omega_t},\quad
 (\partial_t-\Delta_{\omega_t})H_t=\operatorname{tr}_{\omega_t}\theta-n,\quad
 \sqrt{-1}\partial\bar\partial H_t=\omega_t-\theta+t\Ric(\omega_t).
\end{equation}
The potentials $\varphi_t$ converge uniformly to a continuous function
$\varphi_0$. The positive current
$\omega_0=\theta+\sqrt{-1}\partial\bar\partial\varphi_0$ extends the
smooth limiting form, and, for $0<\tau\leq1$,
\begin{equation}\label{eq:finite-time-current}
 \omega_0=\omega_{-\tau}-\tau\Ric(\omega_{-\tau})
       +\sqrt{-1}\partial\bar\partial v_\tau,
 \qquad v_\tau=\varphi_0-H_{-\tau},\quad \|v_\tau\|_\infty\leq C\tau.
\end{equation}
Furthermore, $\theta\leq C\omega_t$.
\end{lemma}


We use the spacetime completion $(Z,d_Z,\mathfrak t)$  introduced in \cite{FangLiStructure}.
By \cite[Theorem~1.1]{LZ26}, the canonical slice $(Z_0,d_0^Z)$ is the
compact geodesic Gromov--Hausdorff limit and contains $X^\circ$ as its
regular part. Its topology agrees with the spacetime topology, since
$d_Z\leq d_0^Z$ and $(Z_0,d_0^Z)$ is compact.
The Type~I bound and \cite[Remark~5.8 and the proofs of Lemmas~5.3--5.4]{LZ26}
give a continuous surjective worldline map $\Phi:X\to Z_0$ with
\begin{equation}\label{eq:finite-time-worldlines}
 d_Z((x,-\tau),\Phi(x))\leq C\sqrt{\tau},\qquad x\in X,
 \quad 0<\tau\leq\tau_0.
\end{equation}

The following proposition was proved by Jian--Song
\cite{JS26b}.
For the reader's convenience, we include a proof using the method
of Proposition~\ref{prop:time-zero-identification} in the shrinker case.

\begin{proposition}\label{prop:finite-time-fibers}
For $x,y\in X$, $\Phi(x)=\Phi(y)$ if and only if $\pi(x)=\pi(y)$.
Consequently, $\Phi$ induces a homeomorphism $Y\simeq Z_0$. Moreover the metric completion of $(Y^\circ, d_{g_0})$ is homeomorphic to $Y$. 
\end{proposition}

\begin{proof}
The proof is similar to that of
Proposition~\ref{prop:time-zero-identification} in the shrinker case.
We record the necessary modifications. By
\cite[Lemma~2.33]{HaconXie}, the fibers of $\pi$ are connected by chains
of compact irreducible curves. For the normalization
$\iota:\tilde C\to X$ of a contracted curve,
\[
 \int_{\tilde C}\iota^*\omega_t
       =(-t)\int_{\tilde C}\iota^*\omega_{-1}\longrightarrow0.
\]
The same Dirichlet-energy and Poincar\'e argument, applied to
$1$-Lipschitz functions on $(Z_0,d_Z)$ extended to $Z$, shows that
$\Phi$ is constant on each fiber; uniform convergence follows from
\eqref{eq:finite-time-worldlines}.

For the converse, take a smooth real function $v$ on $Y$, in the
sense of ambient embeddings, and set $w=v\circ\pi$.
The Schwarz estimate in Lemma~\ref{lem:finite-time-potential} gives
$|\Delta_{\omega_t}w|\leq C_v$. Heat-kernel duality and convergence
to the terminal kernels \cite[Theorem~5.20]{FangLiStructure} yield
\begin{equation}\label{eq:finite-time-base-functions}
 \left|w(x)-\int_Xw\,\d\nu_{\Phi(x);-\tau}\right|\leq C_v\tau.
\end{equation}
If $\Phi(x)=\Phi(y)$, letting $\tau\downarrow0$ gives
$v(\pi(x))=v(\pi(y))$. Such functions separate points of $Y$, so
$\pi(x)=\pi(y)$. The induced continuous bijection $Y\to Z_0$
is a homeomorphism by compactness. The final conclusion follows by the same argument as in Lemma \ref{lem:tangent-distance-modulus} and Lemma \ref{lem:intrinsic-completion}.
\end{proof}

Identify $Y$ with $Z_0$ and write $d_Y=d_0^Z$.
On $Y^\circ$, smooth convergence gives $d_Z\leq d_Y\leq d_{g_0}$.

\begin{lemma}\label{lem:finite-time-moments}
For sufficiently small $\tau_0>0$ and $0<\tau\leq\tau_0$,
\begin{equation}\label{eq:finite-time-estimates}
\begin{gathered}
 \int_0^\tau\int_{X^\circ}\operatorname{tr}_{\omega_{-s}}\omega_0
        \,\d\nu_{p;-s}\,\d s\leq C\tau,\qquad p\in Y,\\
 \int_{X^\circ}d_{g_0}(p,\pi(x))^2\,\d\nu_{p;-\tau}(x)\leq C\tau,
       \qquad p\in Y^\circ.
\end{gathered}
\end{equation}
We have 
\begin{equation}
    d_Z\leq d_Y\leq d_{g_0}\leq Cd_Z.
\end{equation}
\end{lemma}

\begin{proof}
We follow the proofs of Lemmas~\ref{lem:intrinsic-moment}
and~\ref{lem:comparison}, recording the modifications for the finite-time flow.
Set $U_s=H_{-s}-\varphi_0$. By Lemma~\ref{lem:finite-time-potential},
$\|U_s\|_\infty\leq Cs$ and
\[
 (\partial_s+\Delta_{\omega_{-s}})U_s
       =n-\operatorname{tr}_{\omega_{-s}}\omega_0.
\]
Pairing with the terminal kernel and integrating by parts on compact $X$ gives
\[
 \int_\delta^\tau\int_X\operatorname{tr}_{\omega_{-s}}\omega_0
       \,\d\nu_{p;-s}\,\d s
 =n(\tau-\delta)-\int_XU_\tau\,\d\nu_{p;-\tau}
       +\int_XU_\delta\,\d\nu_{p;-\delta}.
\]
Here the trace is taken in the sense of currents. Letting
$\delta\downarrow0$ and using $\|U_s\|_\infty\leq Cs$ gives
the first estimate in \eqref{eq:finite-time-estimates}.

The heat-kernel upper bound \cite[Theorem~7.1]{BamlerHeat}
and the entropy lower bound give a uniform lower bound for $b_s$, where
$\d\nu_{p;-s}=(2\pi s)^{-n}e^{-b_s}\d V_{g_{-s}}$.
The differential Harnack argument in Lemma~\ref{lem:intrinsic-moment}
therefore applies with $\scal_{\omega_{-s}}\leq K/s$.
Then combining it with the scalar curvature lower bound, we obtain there is a uniform $c>0$,
\begin{equation}\label{eq:finite-time-kernel-comparison}
 \nu_{p;-s}\geq c\nu_{p;-t},\qquad
 0<t\leq s\leq4t,\quad s\leq\tau_0.
\end{equation}
Moreover, \eqref{eq:finite-time-base-functions} gives
$\pi_*\nu_{p;-s}\rightharpoonup\delta_p$.
The dyadic argument of Lemma~\ref{lem:intrinsic-moment}, using the
trace estimate and \eqref{eq:finite-time-kernel-comparison}, now proves
the second estimate in \eqref{eq:finite-time-estimates}.
The same argument as in
Lemma~\ref{lem:comparison} gives for $p,q\in Y^\circ$, $$d_{g_0}(p,q)\leq Cd_Z(p,q).$$ Together with $d_Z\leq d_Y\leq d_{g_0}$, this provides an alternative proof of the identification of the metric completion of $(Y^\circ,d_{g_0})$ with $Y$.
\end{proof}

\begin{proof}[Proof of Theorem~\ref{thm:finite-time-completion}]
By Proposition~\ref{prop:finite-time-fibers} and
Lemma~\ref{lem:finite-time-moments}, it suffices to prove $d_{g_0}\leq d_Y$. The tangent-flow structure and smooth regular
convergence in \cite{FangLiStructure} allow us to apply Lemmas~\ref{curve:lengths} and~\ref{curve:orbit} as in the shrinker case. Their regular soliton structure,
intrinsic-distance equality on regular slices, and smooth regular
convergence are as described in Subsection~\ref{subsec:completion-model}.

For
$\omega_{i,t}=r_i^{-2}\omega_{r_i^2t}$, \eqref{eq:finite-time-current} gives
\begin{equation}\label{eq:finite-time-rescaled-current}
 r_i^{-2}\omega_0=\omega_{i,-\delta}
       -\delta\Ric(\omega_{i,-\delta})
       +\sqrt{-1}\partial\bar\partial(r_i^{-2}v_{r_i^2\delta}),
 \qquad \|r_i^{-2}v_{r_i^2\delta}\|_\infty\leq C\delta.
\end{equation}
This is the current estimate used in Proposition~\ref{prop:regular-curve-recovery};
it gives the same holomorphic-strip area and length bounds, while
\eqref{eq:finite-time-worldlines} controls the endpoints. The
metric-differentiation argument in the proof of Theorem~\ref{thm:completion}
then yields $d_{g_0}\leq d_Y$. 
\end{proof}

\begingroup
\small
\setlength{\labelsep}{0.35em}

\endgroup

\vspace{10pt}

Yu Li, Institute of Geometry and Physics, University of Science and
Technology of China, No.~96 Jinzhai Road, Hefei, Anhui Province, 230026,
China; Hefei National Laboratory, No.~5099 West Wangjiang Road, Hefei,
Anhui Province, 230088, China; E-mail: \href{mailto:yuli21@ustc.edu.cn}{yuli21@ustc.edu.cn}.

\vspace{6pt}

Junsheng Zhang, Courant Institute of Mathematical Sciences, New York
University, 251 Mercer St, New York, NY 10012;
E-mail: \href{mailto:jz7561@nyu.edu}{jz7561@nyu.edu}.

\end{document}